\documentclass[11pt]{amsart}

\usepackage{amsmath, amssymb, amsthm}
\usepackage{mathtools}
\usepackage{geometry}
\usepackage{enumitem}
\usepackage{microtype}
\usepackage{tikz}
\usetikzlibrary{calc,intersections}
\usepackage{hyperref}
\usepackage{upref}
\usepackage{zref-clever}
\usepackage{zref-user}
\zcsetup{
    cap,
    nameinlink=false,
    sort=false
}

\allowdisplaybreaks[4]

\theoremstyle{plain}
\newtheorem{thm}{Theorem}[section]
\newtheorem{lem}[thm]{Lemma}
\newtheorem{prop}[thm]{Proposition}
\newtheorem{cor}[thm]{Corollary}

\theoremstyle{definition}
\newtheorem{defi}[thm]{Definition}
\newtheorem{ex}[thm]{Example}
\theoremstyle{remark}
\newtheorem{rem}[thm]{Remark}

\AddToHook{env/thm/begin}{\zcsetup{countertype={thm=theorem}}}
\AddToHook{env/lem/begin}{\zcsetup{countertype={thm=lemma}}}
\AddToHook{env/prop/begin}{\zcsetup{countertype={thm=proposition}}}
\AddToHook{env/cor/begin}{\zcsetup{countertype={thm=corollary}}}
\AddToHook{env/conj/begin}{\zcsetup{countertype={thm=conjecture}}}
\AddToHook{env/obs/begin}{\zcsetup{countertype={thm=observation}}}
\AddToHook{env/cons/begin}{\zcsetup{countertype={thm=construction}}}
\AddToHook{env/ques/begin}{\zcsetup{countertype={thm=question}}}
\AddToHook{env/cla/begin}{\zcsetup{countertype={thm=claim}}}
\AddToHook{env/fact/begin}{\zcsetup{countertype={thm=fact}}}
\AddToHook{env/defi/begin}{\zcsetup{countertype={thm=definition}}}
\AddToHook{env/ex/begin}{\zcsetup{countertype={thm=example}}}
\AddToHook{env/rem/begin}{\zcsetup{countertype={thm=remark}}}
\AddToHook{env/note/begin}{\zcsetup{countertype={thm=note}}}
\AddToHook{env/comm/begin}{\zcsetup{countertype={thm=comment}}}
\numberwithin{equation}{section}
\numberwithin{figure}{section}
\theoremstyle{plain}
\newtheorem*{thm*}{Theorem}
\newtheorem*{cor*}{Corollary}

\newcommand{\Cb}{\mathbb{C}}
\newcommand{\Ib}{\mathbb{I}}
\newcommand{\Rb}{\mathbb{R}}
\newcommand{\Sb}{\mathbb{S}}
\newcommand{\Tb}{\mathbb{T}}
\newcommand{\Zb}{\mathbb{Z}}
\newcommand{\Ac}{\mathcal{A}}
\newcommand{\Bc}{\mathcal{B}}
\newcommand{\Fc}{\mathcal{F}}
\newcommand{\Hc}{\mathcal{H}}
\newcommand{\Oc}{\mathcal{O}}
\newcommand{\Sc}{\mathcal{S}}
\newcommand{\Wc}{\mathcal{W}}
\newcommand{\Sfrk}{\mathfrak{S}}
\newcommand{\abs}[1]{{\left\lvert #1\right\rvert}}
\DeclareMathOperator{\codim}{codim}
\DeclareMathOperator{\rk}{rank}
\DeclareMathOperator{\sgn}{sgn}

\title{Non-vanishing of homotopy groups of Manin--Schechtman arrangements}
\author{Takuya Saito}
\address{Institute for Chemical Reaction Design and Discovery, Hokkaido University}
\email{saito@icredd.hokudai.ac.jp}

\author{So Yamagata}
\address{Department of Applied Mathematics, Fukuoka University}
\email{so.yamagata@fukuoka-u.ac.jp}

\date{}

\begin{document}
\begin{abstract}
  One of the central problems in the topology of hyperplane arrangements is determining whether the complement is a $K(\pi,1)$-space.
  In this paper, we study Manin--Schechtman arrangements, introduced as higher-dimensional analogs of the braid arrangement, and prove that their complements have non-vanishing higher homotopy groups.
  Consequently, these arrangements fail to be $K(\pi,1)$ in a broad range of cases.
\end{abstract}
\maketitle
\section{Introduction}
Let $\Ac$ be a complex hyperplane arrangement in $\Cb^\ell$, and let $M(\Ac) = \Cb^\ell \setminus \bigcup_{H \in \Ac} H$ denote its complement.
One of the central problems in the topology of hyperplane arrangements is determining whether $M(\Ac)$ is a $K(\pi,1)$-space, that is, whether $\pi_p(M(\Ac))$ vanishes for all $p \ge 2$.
If $M(\Ac)$ is $K(\pi,1)$, the arrangement $\Ac$ is called a $K(\pi,1)$ arrangement.
Classical examples of $K(\pi,1)$ arrangements include the braid arrangement, whose complement is the configuration space of distinct points in $\Cb$ \cite{FN62}.
More generally, Brieskorn \cite{Bri73} raised the question of whether Coxeter (reflection) arrangements are $K(\pi,1)$.
The $K(\pi,1)$ property was established for complexified simplicial arrangements by Deligne \cite{Del72}, and for reflection arrangements associated with complex reflection groups by Bessis \cite{Bes15}.
It is also known that fiber-type arrangements, equivalently, supersolvable arrangements, are $K(\pi,1)$ (see \cite{FR87, Ter86}).
More recently, Paolini and Salvetti \cite{PS21} proved the $K(\pi,1)$ property for a large class of arrangements arising from affine Artin groups. On the other hand, many arrangements fail to be $K(\pi,1)$.
Hattori \cite{Hat75} showed that generic arrangements are not $K(\pi,1)$, and further examples have been studied in \cite{PS02, Yos08}.
Concerning structural characterizations, Stanley \cite{Sta07} proved that a graphic arrangement is supersolvable, equivalently, $K(\pi,1)$ if and only if the underlying graph is chordal.
Falk \cite{Fal88, Fal94, Fal95} investigated various topological properties such as tests for the $K(\pi,1)$ property.

Recently, Yoshinaga \cite{Yos24a} introduced a new approach to detecting the non-$K(\pi,1)$ property for the complexifications of real arrangements.
This criterion is formulated in terms of a filtration
\begin{equation*}
  \Sigma_1(\Ac) \supseteq \Sigma_2(\Ac) \supseteq \cdots,
\end{equation*}
which we call the $\Sigma$-\emph{filtration}, where $\Sigma_p(\Ac)$ consists of sign vectors that are consistent on all localizations of codimension at most $p$ (see \zcref{defi:consistency} for details).
A proper inclusion in this filtration implies the existence of a non-trivial higher homotopy group, and hence gives an obstruction to the $K(\pi,1)$ property (see \zcref{sec:consistency}).

In this paper, we study the \emph{Manin--Schechtman arrangements} introduced as higher-dimensional generalizations of the braid arrangement in \cite{MS89}.
These arrangements arise from the space of parallel translates of a fixed generic arrangement and exhibit rich combinatorial and geometric structures.
There are two classes of generic arrangements, depending on whether the combinatorics of the associated Manin--Schechtman arrangement are independent of the choice of the underlying arrangement.
Following \cite{BB97}, we call such arrangements \emph{very generic} if the combinatorics are independent of this choice, and \emph{non-very generic} otherwise.
If the original arrangement $\Ac$ of $n$ hyperplanes in $k$-dimensional space is very generic, then the intersection lattice of its associated Manin--Schechtman arrangement $\Bc(n,k)$ admits a uniform combinatorial description in terms of the poset $P(n,k)$ introduced and studied in \cite{Ath99, BB97}.
Their combinatorics and fundamental groups have been studied in recent work (see, for instance, \cite{ABFKST23,Sai26,Yam23}).

However, much less is known about the higher homotopy-theoretic properties of Manin--Schechtman arrangements.
As pointed out in \cite{FR00}, these arrangements are generally neither free, fiber-type, nor simplicial, making their topology difficult to analyze.
In particular, while almost all rank three cases are known not to be $K(\pi,1)$ \cite{FR87,FR00}, the general case remains open.

The purpose of this paper is to investigate Yoshinaga's $\Sigma$-filtration (\zcref{thm:Yoshinaga}) in the context of Manin--Schechtman arrangements associated with very generic arrangements.
These arrangements provide a natural testing ground for the $\Sigma$-filtration:
they exhibit both cases where the filtration detects nontrivial higher homotopy groups and an exceptional clean case where the obstruction disappears.
Our main result shows that the $\Sigma$-filtration is sufficiently sensitive to detect higher homotopy groups in a broad range of cases.
More precisely, we prove the following.
\begin{thm*}[\zcref{thm:sigma-jump}]
  The following holds:
  \begin{itemize}
    \item $\Sigma_2(\Bc(n,2)) = \Sigma_3(\Bc(n,2))$;
    \item if $k \ge 3$, then there exists a Manin--Schechtman arrangement $\Bc(n,k)$ such that $\Sigma_2(\Bc(n,k)) \supsetneq \Sigma_3(\Bc(n,k))$;
    \item if $n - k \ge 3$, $k \ge 2$, and $3\le p\le n-k-1$, then there exists a Manin--Schechtman arrangement $\Bc(n,k)$ such that $\Sigma_p(\Bc(n,k)) \supsetneq \Sigma_{p+1}(\Bc(n,k))$.
  \end{itemize}
\end{thm*}
Moreover,
\begin{cor*}[\zcref{cor:non-K(pi1)}]
  Suppose that $n - k \ge 3$.
  If $k \ge 3$, then $\pi_p(M(\Bc(n,k)))$ $\ne 0$ for $2 \le p \le n - k - 1$.
  If $k = 2$, then $\pi_p(M(\Bc(n,2))) \ne 0$ for $3 \le p \le n - 3$.
  In particular, $\Bc(n, k)$ is not $K(\pi,1)$ for $k\neq 1$ and $(n,k) \neq (5,2)$.
\end{cor*}
The exceptional case $(n,k) = (5,2)$ illustrates a limitation of the criterion.
Since $\Bc(5,2)$ has rank three, the $\Sigma$-filtration terminates at $\Sigma_3$, and we prove that $\Sigma_2(\Bc(5,2)) = \Sigma_3(\Bc(5,2))$.
Hence, no strict jump occurs, and the non-$K(\pi,1)$ property cannot be detected in this case.

In the non-very generic case, the situation is more subtle.
Some Manin--Schechtman arrangements are known to be $K(\pi,1)$ (see \zcref{ex:H3}), indicating that the topology depends strongly on the geometry of the underlying arrangement.

This paper is organized as follows.
In \zcref{sec:consistency}, we review the $\Sigma$-filtration and Yoshinaga's criterion for non-vanishing of homotopy groups.
In \zcref{sec:MSnk-BBA-poset}, we recall the definition and combinatorial structure of Manin--Schechtman arrangements.
\zcref{sec:main_result} presents an outline of the proof of the main theorem stated above.
In \zcref{sec:pfoflem41}, we treat the exceptional case $\Bc(5,2)$ and show its cleanliness, i.e., $\Sigma_2(\Bc(5,2)) = \Sigma_3(\Bc(5,2))$.
Finally, in \zcref{sec:non-very}, we study certain subvarieties in the parameter space of generic arrangements in order to prove the remaining cases $\Bc(n,k)$ for $k \ge 2$, $n - k \ge 3$, and $(n,k) \neq (5,2)$.
We also examine the relationship with matroid strata.

\subsection*{Notation}
The terminology for hyperplane arrangements follows \cite{OT92}.
The essentialization of an arrangement $\Ac$ is denoted by $\Ac^{\mathrm{ess}}$.
The intersection lattice of $\Ac$, denoted by $L(\Ac)$, is the graded lattice of intersections of hyperplanes in $\Ac$, ordered by reverse inclusion.
We denote by $L_p(\Ac)$ the set of rank $p$ elements of $L(\Ac)$.
For $X \in L(\Ac)$, we denote by $\Ac_X$ the localization of $\Ac$ at $X$.

For a positive integer $n$, we write $[n] \coloneq \{ 1, \dots, n \}$.
We denote by $2^{[n]}$ the power set of $[n]$, and by $\binom{[n]}{k} \coloneq \{ I \in 2^{[n]} \mid \abs{I} = k \}$ the set of $k$-element subsets of $[n]$.
For brevity, when no confusion is likely to arise, we sometimes denote the set $\{ a, b, \dots, c \}$ simply as $ab \dots c$.
\section*{Acknowledgments}
The authors would like to thank Masahiko Yoshinaga for pointing out the relevance of the $\Zb^3$-test discussed in \zcref{rem:Z^3}.
The second author was supported by JSPS KAKENHI Grant Numbers JP24K16926 and JPJSBP120256504.
\section{Consistency and $\Sigma$-filtration}\label{sec:consistency}
Let $\Ac$ be a central arrangement in $V = \Rb^\ell$, and for each $H \in \Ac$ fix a linear form $\alpha_H \in V^*$ such that $H = \ker(\alpha_H)$.
Let $\varepsilon \colon \Ac \to \{ \pm \}$ be a sign vector and set $\varepsilon_H = \varepsilon(H)$ for $H \in \Ac$. Define the associated system of half-spaces by
\begin{equation*}
  \Hc(\varepsilon) = (H^{\varepsilon_H} \mid H \in \Ac), \text{ where }H^{\pm} \coloneq \{ x \in V \mid \pm \alpha_H(x) > 0 \}.
\end{equation*}
For a subarrangement $\Ac' \subset \Ac$, we denote by $\varepsilon \vert_{\Ac'}$ the restriction of $\varepsilon$ to $\Ac'$.

\begin{defi}\label{defi:consistency}
  Let $X \in L(\Ac)$.
  We say that the system $\Hc(\varepsilon)$ is \emph{consistent at $X$} if
  \begin{equation*}
    \bigcap_{H \in {\Ac_X}} H^{\varepsilon_H}
    \neq \emptyset.
  \end{equation*}
\end{defi}
For $1 \le p \le \ell$, let $\Sigma_p(\Ac)$ denote the set of sign vectors $\varepsilon$ such that $\Hc(\varepsilon)$ is consistent at every $X \in L(\Ac)$ with $\codim(X) \le p$.
\zcref{fig:consistency} illustrates examples of consistent and inconsistent systems, where the shaded region indicates the intersection of the chosen half-spaces.
\begin{figure}[htbp]
  \centering
  \begin{tikzpicture}[scale=0.7]
    \coordinate (0) at (30:2);
    \coordinate (1) at (90:2);
    \coordinate (2) at (150:2);
    \coordinate (3) at (210:2);
    \coordinate (4) at (270:2);
    \coordinate (5) at (330:2);
    \draw (0) -- (3);
    \draw (1) -- (4);
    \draw (2) -- (5);
    \fill[blue!20] (0,0) -- (90:2) -- (150:2) -- cycle;
    \draw[->] (30:1.5) -- (60:1.5);
    \draw[->] (90:1.5) -- (120:1.5);
    \draw[->] (330:1.5) -- (0:1.5);
  \end{tikzpicture}
  \hspace{0.1\columnwidth}
  \begin{tikzpicture}[scale=0.7]
    \coordinate (0) at (30:2);
    \coordinate (1) at (90:2);
    \coordinate (2) at (150:2);
    \coordinate (3) at (210:2);
    \coordinate (4) at (270:2);
    \coordinate (5) at (330:2);
    \draw (0) -- (3);
    \draw (1) -- (4);
    \draw (2) -- (5);
    \draw[->] (30:1.2) -- (0:1.2);
    \draw[->] (90:1.5) -- (120:1.5);
    \draw[->] (330:1.5) -- (0:1.5);
  \end{tikzpicture}
  \caption{Consistent (left) and inconsistent (right) systems of half-spaces for the braid arrangement $\mathrm{Br}(3)$.}
  \label{fig:consistency}
\end{figure}
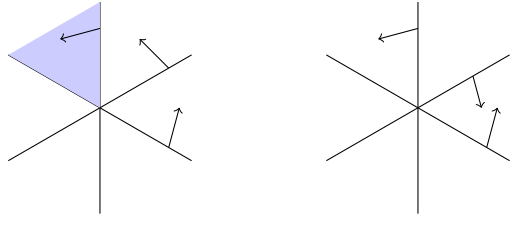

\begin{equation*}
  \Sigma_1(\Ac) \supseteq \cdots \supseteq \Sigma_\ell(\Ac).
\end{equation*}
Yoshinaga \cite{Yos24a} showed that a proper inclusion in this filtration detects the non-vanishing of higher homotopy groups of the complement of arrangements.

\begin{thm}[{\cite[Theorem~5.1]{Yos24a}}]\label{thm:Yoshinaga}
  Let $\Ac$ be a central and essential arrangement in $V = \Rb^\ell$.
  \begin{itemize}
    \item If $\Sigma_p(\Ac) \supsetneq \Sigma_{p+1}(\Ac)$ for some $p \ge 2$, then $\pi_p(M(\Ac \otimes \Cb)) \neq 0$.
    \item If $M(\Ac \otimes \Cb)$ is $K(\pi, 1)$, then $\Sigma_2(\Ac) = \cdots = \Sigma_\ell(\Ac)$.
  \end{itemize}
  Here $\Ac \otimes \Cb \coloneq \{ H \otimes \Cb \subset \Cb^\ell \mid H \in \Ac \}$ denotes the complexification of $\Ac$.
\end{thm}
An arrangement $\Ac$ is called \emph{clean} in \cite{DDBP25} if $\Sigma_2(\Ac) = \Sigma_\ell(\Ac)$, which is a necessary condition for the $K(\pi,1)$ property.

The following lemmas follow immediately from the definition of consistency.
\begin{lem}\label{lem:sigma-essentialization}
  Let $\Ac$ be a central arrangement.
  Then, for every $p$ we have
  \begin{equation*}
    \Sigma_p(\Ac^{\mathrm{ess}}) = \Sigma_p(\Ac).
  \end{equation*}
\end{lem}

\begin{lem}\label{lem:sigma-product-decomposition}
  Let $\Ac_1$ and $\Ac_2$ be central arrangements of ranks $\ell_1$ and $\ell_2$, respectively.
  Then, for every $p$ we have
  \begin{equation*}
    \Sigma_p(\Ac_1 \oplus \Ac_2) = \Sigma_{\min(p,\ell_1)}(\Ac_1) \times \Sigma_{\min(p,\ell_2)}(\Ac_2).
  \end{equation*}
\end{lem}
\begin{proof}
  Write $\Ac = \Ac_1 \oplus \Ac_2$ as an arrangement in $V_1 \oplus V_2$.
  A sign vector $\varepsilon$ on $\Ac$ restricts to sign vectors $\varepsilon\vert_{\Ac_1} \in \{ \pm \}^{\Ac_1}$ and $\varepsilon\vert_{\Ac_2} \in \{ \pm \}^{\Ac_2}$.
  For $X = X_1 \oplus X_2$ with $X_i \in L(\Ac_i)$, we have $\Ac_X = (\Ac_1)_{X_1} \oplus (\Ac_2)_{X_2}$ and $\codim(X) = \codim(X_1) + \codim(X_2)$.
  Hence consistency at $X$ is equivalent to consistency at $X_1$ and $X_2$, thus $\varepsilon \in \Sigma_p(\Ac)$ if and only if $\varepsilon\vert_{\Ac_i} \in \Sigma_{\min(p,\ell_i)}(\Ac_i)$ for $i = 1, 2$.
\end{proof}

\begin{lem}\label{lem:extend}
  Let $\Ac$ be a central arrangement and let $X\in L(\Ac)$.
  For every $p$, if $\Sigma_p(\Ac)=\Sigma_{p+1}(\Ac)$, then $\Sigma_p(\Ac_X)=\Sigma_{p+1}(\Ac_X)$ holds.
\end{lem}
\begin{proof}
  Under the assumption $\Sigma_p(\Ac)=\Sigma_{p+1}(\Ac)$, we show that if $\varepsilon'\in \Sigma_p(\Ac_X)$ then $\varepsilon'\in \Sigma_{p+1}(\Ac_X)$.
  Choose $x \in X \setminus \bigcup_{H \in \Ac \setminus \Ac_X} H$ and we define a sign vector $\varepsilon$ on $\Ac$ by $\varepsilon_H =\varepsilon'_H$ if $H \in \Ac_X$ and $\varepsilon_H =\operatorname{sgn} \alpha_H(x)$ if $H \notin \Ac_X$.
  It is easy to check that $\varepsilon \in \Sigma_p(\Ac)$.
  Thus, we have $\varepsilon \in \Sigma_{p+1}(\Ac)$ and this implies $\varepsilon'=\varepsilon|_{\Ac_X}\in \Sigma_{p+1}(\Ac_X)$.
\end{proof}
\begin{cor}
  If $\Ac$ is clean, then so is every localization $\Ac_X$ for $X\in L(\Ac)$.
\end{cor}

For later use in \zcref{sec:pfoflem41}, we reformulate consistency in terms of positive linear relations among the defining forms.
For $X \in L(\Ac)$,
consistency at $X$ is equivalent to the existence of $x$ satisfying
\begin{equation*}
  \varepsilon_H \alpha_H(x) > 0 \quad \text{for all } H \in \Ac_X.
\end{equation*}
Let $B_X(\varepsilon)$ be the matrix whose rows are $\varepsilon_H \alpha_H$.
Then this condition can be written as
\begin{equation*}
  B_X(\varepsilon) x > 0.
\end{equation*}
Here, all inequalities are understood entry-wise.

The non-existence of a solution to such a system is characterized by the following theorem of the alternative.
\begin{thm}[Gordan theorem \cite{BLVS99}]\label{thm:gordan}
  Let $B \in \Rb^{m \times \ell}$.
  Exactly one of the following holds:
  \begin{itemize}
    \item there exists $u \in \Rb^\ell$ with $Bu > 0$;
    \item there exists $\lambda \in \Rb^m$ with $\lambda \ge 0$, $\lambda \ne 0$, and $B^{\top} \lambda = 0$.
  \end{itemize}
\end{thm}

Applying \zcref{thm:gordan} to $B_X(\varepsilon)$, the system $B_X(\varepsilon) x > 0$ admits no solution if and only if there exists a nonzero vector $\lambda \in \Rb^{\Ac_X}_{\ge 0}$ such that
\begin{equation*}
  \sum_{H \in \Ac_X} \lambda_H \varepsilon_H \alpha_H = 0.
\end{equation*}
After removing zero coefficients, we may assume that all $\lambda_H > 0$ on the support.
We call such a positive linear relation a \emph{Gordan certificate}.
Note that, using the criterion stated in \zcref{thm:Yoshinaga} together with localization to rank three elements and \zcref{thm:gordan}, we can recover an analog of \cite[Corollary~5.2]{Fal88}.

Furthermore, we provide an additional criterion for determining the proper inclusion of $\Sigma$-filtration.
\begin{defi}
  Let $\Ac$ be a central arrangement in $\Rb^\ell$.
  We call a chamber $C \subset \Rb^\ell$ a \emph{simple chamber} if the closure $\overline{C}$ of $C$ is the product of a simplicial cone and the center of $\Ac$, and every hyperplane $H \in \Ac$ satisfying $H \cap \overline{C} \neq \{ 0 \}$ is a wall of $C$.
\end{defi}
\begin{prop}\label{prop:generic-chamber}
  Let $\Ac$ be a real central arrangement of rank $\ell$
  with at least $\ell + 1$ hyperplanes.
  If $\Ac$ has a simple chamber, then $\Sigma_{\ell-1}(\Ac) \supsetneq \Sigma_\ell(\Ac)$ holds.
\end{prop}
\begin{proof}
  Without loss of generality, suppose that $\Ac$ is essential.
  For any $x \in M(\Ac)$, we define $\varepsilon(x) \in \{ \pm \}^{\Ac}$ as $\varepsilon(x)_H \coloneq \sgn \alpha_H(x)$.
  It follows immediately that $\varepsilon \in \Sigma_\ell(\Ac)$ if and only if $\varepsilon = \varepsilon(x)$ for some $x \in M(\Ac)$.

  Let $\Wc \subset \Ac$ be the set of walls of the simple chamber $C$.
  For $x \in C$, define $\varepsilon^{C} \in \{ \pm \}^\Ac$ by setting $\varepsilon_H^C = \varepsilon(x)_H$ if $H \in \Ac \setminus \Wc$ and $\varepsilon_H^C = -\varepsilon(x)_H$ if $H \in \Wc$.
  Then $\varepsilon^C \notin \Sigma_\ell(\Ac)$ holds.

  Since $C$ is simplicial, we have $\abs{\Wc} = \ell$ and $\bigcap_{H \in \Wc} H = \{ 0 \}$.
  Thus, $\abs{\Ac_X \cap \Wc} \le \ell - 1$ holds for any $X \in L_{\ell - 1}(\Ac)$.
  In other words, for any $X \in L_{\ell - 1}(\Ac)$, there exists $H \in \Wc$ such that $\Ac_X \subset \Ac \setminus H$.
  Therefore, it is enough to show $\varepsilon^C \vert_{\Ac \setminus H} \in \Sigma_\ell(\Ac \setminus H)$ for any $H \in \Wc$.

  Let $C_\delta$ be the $\delta$-neighborhood $\{ x \in \Rb^{\ell} \mid d(x,C) \le \delta, \lVert x \rVert = 1 \}$ of $C$ on the unit sphere for sufficiently small $\delta > 0$.
  For $x \in M(\Ac) \cap C_{\delta}$ and any $H \in \Ac \setminus \Wc$, $\varepsilon(x)_H$ is independent of the choice of $x$, that is, $\varepsilon(x)\vert_{\Ac \setminus \Wc} = \varepsilon^C\vert_{\Ac \setminus \Wc}$.
  On the other hand, $\{ \varepsilon(x) \vert_\Wc \mid x \in M(\Ac) \cap C_{\delta} \} = \{ \pm \}^\Wc \setminus \{ \varepsilon^{C} \vert_\Wc \}$ holds.
  In particular, for any $H \in \Wc$, we have $\{ \varepsilon(x)\vert_{\Wc \setminus \{ H \}} \mid x \in M(\Ac) \cap C_\delta \} = \{ \pm \}^{\Wc \setminus \{ H \}}$.
  Therefore, for $H \in \Wc$, any $\varepsilon \in \{\pm\}^{\Ac}$ with $\varepsilon\vert_{\Ac \setminus \Wc} = \varepsilon^C\vert_{\Ac \setminus \Wc}$ satisfies $\varepsilon\vert_{\Ac \setminus H} \in \Sigma^{\Ac \setminus H}_\ell$.
  In particular, we have $\varepsilon^C\vert_{\Ac \setminus H} \in \Sigma_\ell({\Ac \setminus H})$ for any $H \in \Wc$.
\end{proof}

If $\Ac$ is an arrangement obtained by perturbing an intersection of multiplicity $\ell$ and rank $\ell - 1$ into $r$ distinct intersections, then $\Ac$ admits a simple chamber.
In terms of matroid theory, a simple chamber arises when the relaxation of a circuit-hyperplane (cf. \cite{Oxl11}) in the matroid associated with $\Ac$ is realized by perturbing the real hyperplane arrangement.

\begin{cor}\label{cor:perturbation-method}
  Let $\Ac'$ be an arrangement of rank $\ell$ with a rank $\ell - 1$ intersection $X$, and suppose that $\abs{\Ac'_X} = \ell$.
  If an arrangement $\Ac$ is obtained from $\Ac'$ by perturbing $X$ into $r$ distinct intersections, then $\Sigma_{\ell-1}(\Ac) \supsetneq \Sigma_\ell(\Ac)$ holds.
\end{cor}
\begin{proof}
  Without loss of generality, suppose that $\Ac$ is essential.
  Since $\abs{\Ac'_X} = \ell$, if $X$ is perturbed into $r$ distinct intersections, then $\Ac'_X$ deforms into a Boolean arrangement.
  Furthermore, the localization at each intersection arising from perturbing an arbitrary intersection $Y$ is a subarrangement of $\Ac_Y$.
  Therefore, $\Ac$ has a simple chamber and hence $\Sigma_{\ell-1}(\Ac) \supsetneq \Sigma_\ell(\Ac)$ holds by \zcref{prop:generic-chamber}.
\end{proof}

\begin{ex}\label{ex:B63}
  Consider a real arrangement $\Bc$ of rank three illustrated in \zcref{fig:ex-generic-chamber}, which is spanned by six vectors in $\Rb^3$.
  As will be explained in \zcref{sec:MSnk-BBA-poset}, this is equivalent to the Manin--Schechtman arrangement $\Bc(6, 3)$.

  In the figure, the shaded chamber $C$ containing the point $x$ is simple.
  Also, if we take a sign vector $\varepsilon^C \in \{ \pm \}^\Bc$ as defined in the proof of \zcref{prop:generic-chamber}, the associated system $\Hc(\varepsilon^C)$ of half-spaces corresponds to the directions indicated by the arrows in \zcref{fig:ex-generic-chamber}.
  We can check that $\varepsilon^C$ is consistent at every intersection of rank two, but we have $\bigcap_{H \in \Bc} H^{\varepsilon^C_H} = \emptyset$.
  Therefore, $\Sigma_2(\Bc) \supsetneq \Sigma_3(\Bc)$ holds.
  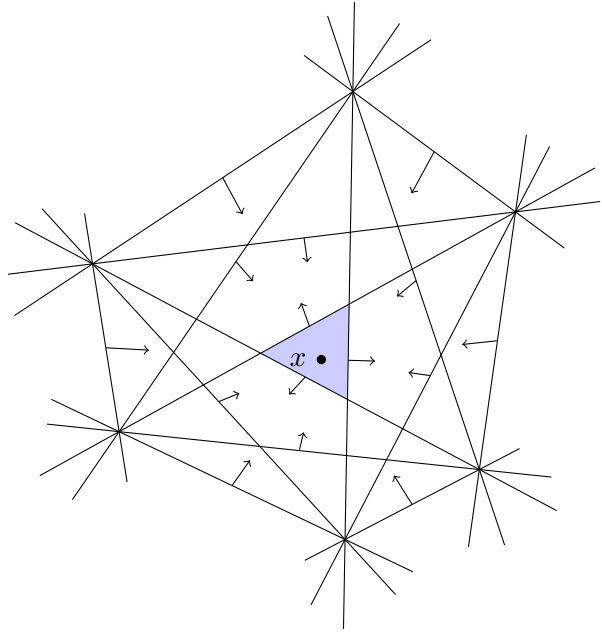
\begin{figure}[htbp]
    \centering
    \begin{tikzpicture}[scale=3]
      \coordinate (0) at (27:1);
      \coordinate (1) at (80:1);
      \coordinate (2) at (167:1);
      \coordinate (3) at (211:1);
      \coordinate (4) at (278:1);
      \coordinate (5) at (317:1);
      \coordinate (x) at (intersection of 0--3 and 1--4);
      \coordinate (y) at (intersection of 2--5 and 1--4);
      \coordinate (z) at (intersection of 0--3 and 2--5);
      \fill[blue!20] (x) -- (y) -- (z);
      \node[inner sep=0.1em, fill, draw, circle, label=left:{$x$}] (p) at (-80:0.2) {};
      \draw ($(0) !1.3! (1)$) -- ($(1) !1.3! (0)$);
      \draw ($(1) !1.3! (2)$) -- ($(2) !1.3! (1)$);
      \draw ($(2) !1.3! (3)$) -- ($(3) !1.3! (2)$);
      \draw ($(3) !1.3! (4)$) -- ($(4) !1.3! (3)$);
      \draw ($(4) !1.3! (5)$) -- ($(5) !1.3! (4)$);
      \draw ($(0) !1.3! (5)$) -- ($(5) !1.3! (0)$);
      \draw ($(4) !1.2! (0)$) -- ($(0) !1.2! (4)$);
      \draw ($(0) !1.2! (2)$) -- ($(2) !1.2! (0)$);
      \draw ($(0) !1.2! (3)$) -- ($(3) !1.2! (0)$);
      \draw ($(1) !1.2! (3)$) -- ($(3) !1.2! (1)$);
      \draw ($(1) !1.2! (4)$) -- ($(4) !1.2! (1)$);
      \draw ($(2) !1.2! (4)$) -- ($(4) !1.2! (2)$);
      \draw ($(1) !1.2! (5)$) -- ($(5) !1.2! (1)$);
      \draw ($(2) !1.2! (5)$) -- ($(5) !1.2! (2)$);
      \draw ($(3) !1.2! (5)$) -- ($(5) !1.2! (3)$);
      \draw[->] ($(0) !0.5! (1)$) -- ($($(0) ! 0.5 ! (1) $) !0.2! (p)$ );
      \draw[->] ($(0) !0.5! (2)$) -- ($($(0) ! 0.5 ! (2) $) !0.2! (p)$ );
      \draw[->] ($(0) !0.52! (3)$) -- ($($(0) ! 0.52 ! (3) $) !-0.7! (p)$ );
      \draw[->] ($(0) !0.5! (4)$) -- ($($(0) ! 0.5 ! (4) $) !0.2! (p)$ );
      \draw[->] ($(0) !0.5! (5)$) -- ($($(0) ! 0.5 ! (5) $) !0.2! (p)$ );
      \draw[->] ($(1) !0.5! (2)$) -- ($($(1) ! 0.5 ! (2) $) !0.2! (p)$ );
      \draw[->] ($(1) !0.5! (3)$) -- ($($(1) ! 0.5 ! (3) $) !0.2! (p)$ );
      \draw[->] ($(1) !0.6! (4)$) -- ($($(1) ! 0.6 ! (4) $) !-1.0! (p)$ );
      \draw[->] ($(1) !0.5! (5)$) -- ($($(1) ! 0.5 ! (5) $) !0.2! (p)$ );
      \draw[->] ($(2) !0.5! (3)$) -- ($($(2) ! 0.5 ! (3) $) !0.2! (p)$ );
      \draw[->] ($(2) !0.5! (4)$) -- ($($(2) ! 0.5 ! (4) $) !0.2! (p)$ );
      \draw[->] ($(2) !0.55! (5)$) -- ($($(2) ! 0.55 ! (5) $) !-1! (p)$ );
      \draw[->] ($(3) !0.5! (4)$) -- ($($(3) ! 0.5 ! (4) $) !0.2! (p)$ );
      \draw[->] ($(3) !0.5! (5)$) -- ($($(3) ! 0.5 ! (5) $) !0.2! (p)$ );
      \draw[->] ($(4) !0.5! (5)$) -- ($($(4) ! 0.5 ! (5) $) !0.2! (p)$ );
    \end{tikzpicture}
    \caption{An inconsistent sign system for the arrangement of rank three determined by six generic vectors in $\Rb^3$, combinatorially equivalent to $\Bc(6,3)$. The shaded region represents a simple chamber.}\label{fig:ex-generic-chamber}
  \end{figure}
\end{ex}

\begin{rem}\label{rem:Z^3}
      The rank-three arrangement in \zcref{ex:B63} can also be shown to be non-$K(\pi,1)$ by the classical $\Zb^3$-test/simple-triangle obstruction.
      In the affine line-arrangement setting, this obstruction is based on the fact that if the fundamental group contains a free abelian subgroup of rank three, then the arrangement is not $K(\pi,1)$
      (\cite{Hat75,FR87}, see also \cite[Proposition 8.5.]{Yos24b}).
      A \emph{simple triangle} in the sense of Edelman--Reiner \cite{ER95} is a configuration of three non-concurrent lines such that no other line of the arrangement passes through any of their pairwise intersection points.
      The existence of a simple triangle is a sufficient condition for the $\Zb^3$-test.
      In \zcref{ex:B63}, the simple triangle is precisely the triangle corresponding to the simple chamber in the affine arrangement obtained by deconing $\Bc$.
\end{rem}

\section{Manin--Schechtman arrangements and their combinatorics}\label{sec:MSnk-BBA-poset}
\subsection{Definition and Examples}
We begin by recalling the definition of the Manin--Schechtman arrangements~\cite{MS89}.
Fix a generic hyperplane arrangement $\Ac^0 = \{ H_1^0, \dots, H_n^0 \}$ with normal vectors $\alpha_i \in \Cb^k$, that is, $\codim \bigcap_{i \in I} H_i = \min \{ \abs{I}, k \}$ for any $I \subset [n]$.

Consider the space of all parallel translates of $\Ac^0$:
\begin{equation*}
  \Sb = \Sb(\Ac^0) \coloneq \{ (H_1^{t_1}, \dots, H_n^{t_n}) \mid t_1, \dots, t_n \in \Cb \},
\end{equation*}
where $H_i^{t_i} \coloneq \alpha_i^{-1}(t_i)$.

For each subset $I \subset [n]$ with $\abs{I} = k + 1$, define
\begin{equation*}
  D_I = D_I(\Ac^0) \coloneq \left\{ (H_1^{t_1}, \dots, H_n^{t_n}) \in \Sb \
  \middle| \ \bigcap_{i \in I} H_i^{t_i} \neq \emptyset \right \}.
\end{equation*}
Equivalently, $D_I$ consists of those translates for which the subarrangement indexed by $I$ fails to be generic.
Each $D_I$ is a hyperplane in $\Sb \simeq \Cb^n$.
The \emph{Manin--Schechtman arrangement} is the hyperplane arrangement
\begin{equation*}
  \Bc(n,k,\Ac^0) \coloneq \left\{ D_I \mathrel{}\middle|\mathrel{} I \in \binom{[n]}{k+1} \right\},
\end{equation*}
introduced in~\cite{MS89} as a higher analog of the braid arrangement.
In particular, $\Bc(n,1,\Ac^0)$ coincides with the braid arrangement $\mathrm{Br}(n) = \{ \{ x_i = x_j \} \subset \Cb^n \mid 1 \le i < j \le n \}$.

For simplicity, the intersection $\bigcap_{I \in \binom{T}{k+1}} D_I$ of $D_I = D_I(\Ac^0)$, $I \subset T$ will be denoted by $D_T = D_T(\Ac^0)$ for any $T \subset [n]$ with $\abs{T} \ge k + 1$.
Then the center of the Manin--Schechtman arrangement $\Bc(n,k,\Ac^0)$ is $D_{[n]}$.
Furthermore, we have
\begin{align*}
  D_T & = \left\{ (H_1^{t_1}, \dots, H_n^{t_n}) \in \Sb \mathrel{}\middle|\mathrel{} \bigcap_{i \in T} H_i^{t_i} \neq \emptyset \right\}         \\
      & = \left\{ (H_1^{t_1}, \dots, H_n^{t_n}) \in \Sb \mathrel{}\middle|\mathrel{} t_i = \alpha_i(x), x \in \Cb^k \text{ if } i \in T \right\}
\end{align*}
and $\codim D_T = \abs{T} - k$.
Hence, the rank of $\Bc(n, k, \Ac^0)$ is $\codim D_{[n]} = n - k$ and Manin--Schechtman arrangements are not essential.

For each $(k + 1)$-subset $I = \{ i_1 < \dots < i_{k+1} \} \subset [n]$, the normal vector $\alpha_I \in \Sb^\ast$ of $D_I$ is given by
\begin{equation*}
  (\alpha_I)_j =
  \begin{cases}
    (-1)^{p-1} \det(\alpha_{i_1}, \dots, \alpha_{i_{p-1}}, \alpha_{i_{p+1}}, \dots, \alpha_{i_{k+1}}) & \text{if } j = i_p,    \\
    0                                                                                                 & \text{if } j \notin I.
  \end{cases}
\end{equation*}
This follows immediately from the following linear relation
\begin{equation*}
  \sum_{p=1}^{k+1}(\alpha_I)_{i_p} \alpha_{i_p} = 0.
\end{equation*}

\begin{ex}
  Let $\Ac^0$ be the generic arrangement in $\Cb^2$ with normal vectors $(\alpha_1\ \dots\ \alpha_4) =
    \begin{pmatrix}
      1 & 0 & 1 & -1 \\
      0 & 1 & 1 & 1
    \end{pmatrix}.$
  Then $\Bc(4,2,\Ac^0) = \{ D_{123}, D_{124}, D_{134}, D_{234} \}$ whose normal vectors are
  \begin{equation*}
    \begin{pmatrix}
      \alpha_{123} \\
      \alpha_{124} \\
      \alpha_{134} \\
      \alpha_{234}
    \end{pmatrix}=
    \begin{pmatrix}
      -1 & -1 & 1  & 0  \\
      1  & -1 & 0  & 1  \\
      2  & 0  & -1 & 1  \\
      0  & 2  & -1 & -1
    \end{pmatrix}.
  \end{equation*}
\end{ex}
\begin{rem}
  The essentialization of the Manin--Schechtman arrangement $\Bc(n,k,\Ac^0)$ can also be realized as a hyperplane arrangement determined by $n$ points in general position in $\Cb^{n-k}$ as shown by Falk \cite{Fal94}.

  Since the subspace $D_T$ can be written as
  \begin{align*}
    D_T & = \left\{(H_1^{t_1}, \dots, H_n^{t_n}) \in \Sb \mathrel{}\middle|\mathrel{} t_i = \alpha_i(x) \ (x \in \Cb^k) \text{ for all } i \in T \right\} \\
        & = \operatorname{Im} \left( (\alpha_1, \dots, \alpha_n) \colon \Cb^k \to \Sb(\Ac^0) \right) + \Sb(\{ H_i \in \Ac^0 \mid i \notin T \})           \\
        & = D_{[n]} + \Sb(\{H_i \in \Ac^0 \mid i \notin T \}) ,
  \end{align*}
  the essentialization of $\Bc(n,k,\Ac^0)$ is linearly equivalent to an arrangement spanned by $\overline{\mathbf{e}}_i$, $i \in [n]$ in $\Sb(\Ac^0) / D_{[n]}$, where $\overline{\mathbf{e}}_i = \mathbf{e}_{i} + D_{[n]}$ and the vectors $\mathbf{e}_{i}$ are the canonical basis vectors of $\Sb(\Ac^0)$.

  By this observation, the essentialization of $\Bc(n,k,\Ac^0)$ is linearly equivalent to the arrangement spanned by $\beta_1, \dots, \beta_n \in \Cb^{n-k}$ where $\beta_1,\dots,\beta_n$ form a Gale dual of the normal vectors $\alpha_1, \dots, \alpha_n$ of $\Ac^0$, that is, the row space of the matrix $(\beta_1 \dots \beta_n)$ is orthogonal to that of the matrix $(\alpha_1 \dots \alpha_n)$.
\end{rem}
Using the above description, we give the following example.
\begin{ex}\label{ex:H3}
  Let $\phi = \frac{1+\sqrt{5}}{2}$ be the golden ratio and let $\Ac_D^0$ be the arrangement whose normal vectors are
  \begin{equation*}
    (\alpha_1 \dots \alpha_n) =
    \begin{pmatrix}
      1    & 1     & 0    & 0     & \phi & -\phi \\
      0    & 0     & \phi & -\phi & 1    & 1     \\
      \phi & -\phi & 1    & 1     & 0    & 0
    \end{pmatrix}
  \end{equation*}
  Then the essential part of $\Bc(6,3,\Ac_D^0)$ is linearly equivalent to the arrangement shown in \zcref{fig:H3} and spanned by the six column vectors in the following matrix
  \begin{equation*}
    \begin{pmatrix}
      -1   & 1    & \phi & \phi & 0    & 0    \\
      0    & 0    & -1   & 1    & \phi & \phi \\
      \phi & \phi & 0    & 0    & -1   & 1
    \end{pmatrix}.
  \end{equation*}
  Furthermore, $\Bc(6,3,\Ac_D^0)$ has the intersection lattice of minimal cardinality for $(n,k) = (6,3)$ over the real numbers (see \cite{SS25}), and is also linearly equivalent to the Coxeter arrangement of type $H_3$.
  Alternatively, we can confirm directly that $\Bc(6,3,\Ac_D^0)$ is a simplicial arrangement.
  Therefore, $\Bc(6,3,\Ac_D^0)$ is a $K(\pi,1)$ arrangement.
  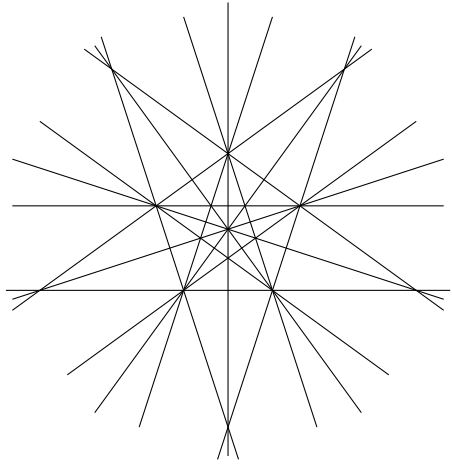
\begin{figure}[h]
    \centering
    \begin{tikzpicture}[scale=1]
      \coordinate (0) at (18:1);
      \coordinate (1) at (90:1);
      \coordinate (2) at (162:1);
      \coordinate (3) at (234:1);
      \coordinate (4) at (306:1);
      \coordinate (5) at (0:0);
      \draw ($(0) !3! (1)$) -- ($(1) !3! (0)$);
      \draw ($(1) !3! (2)$) -- ($(2) !3! (1)$);
      \draw ($(2) !3! (3)$) -- ($(3) !3! (2)$);
      \draw ($(3) !3! (4)$) -- ($(4) !3! (3)$);
      \draw ($(4) !3! (0)$) -- ($(0) !3! (4)$);
      \draw ($(0) !2! (2)$) -- ($(2) !2! (0)$);
      \draw ($(0) !2! (3)$) -- ($(3) !2! (0)$);
      \draw ($(1) !2! (3)$) -- ($(3) !2! (1)$);
      \draw ($(1) !2! (4)$) -- ($(4) !2! (1)$);
      \draw ($(2) !2! (4)$) -- ($(4) !2! (2)$);
      \draw ($(0) !4! (5)$) -- ($(5) !3! (0)$);
      \draw ($(1) !4! (5)$) -- ($(5) !3! (1)$);
      \draw ($(2) !4! (5)$) -- ($(5) !3! (2)$);
      \draw ($(3) !4! (5)$) -- ($(5) !3! (3)$);
      \draw ($(4) !4! (5)$) -- ($(5) !3! (4)$);
    \end{tikzpicture}
    \caption{The projectivization of $\Bc(6,3,\Ac_D^0)$ with the smallest possible intersection lattice, linearly equivalent to the Coxeter arrangement $H_3$}
    \label{fig:H3}
  \end{figure}
\end{ex}

\subsection{Intersection lattice}
The intersection lattice of $\Bc(n,k,\Ac^0)$ may depend on the choice of the underlying generic arrangement $\Ac^0$ (see \cite{Fal94}).
For fixed $n$ and $k$, Bayer and Brandt~\cite{BB97} conjectured that the largest intersection lattice of $\Bc(n,k)$ is described by a certain combinatorial poset $P(n,k)$.
This conjecture was later proved by Athanasiadis~\cite{Ath99}, who showed that, for a very generic arrangement $\Ac^0$, the intersection lattice of $\Bc(n,k,\Ac^0)$ is isomorphic to $P(n,k)$.

\begin{defi}
  We define the poset $P(n,k)$ as follows.
  The elements of $P(n,k)$ are finite collections $\Tb$ of subsets of $[n]$ satisfying:
  \begin{itemize}
    \item[(i)] $\abs{T} \ge k + 1$ for every $T \in \Tb$ ;
    \item[(ii)] $\Tb$ is an antichain, i.e., for distinct elements $T$ and $T'$ in $\Tb$ neither $T \subset T'$ nor $T \supset T'$ holds ;
    \item[(iii)] for every subset $\Tb' \subset \Tb$ with $\abs{\Tb'} \ge 2$ we have
          \begin{equation*}
            \abs{\bigcup_{T \in \Tb'} T} - k > \sum_{T \in \Tb'}(\abs{T} - k) .
          \end{equation*}
  \end{itemize}
  The partial order on $P(n, k)$ is given by
  \begin{equation*}
    \Tb \preceq \Tb' \quad \Longleftrightarrow \quad \text{for every } T \in \Tb \text{ there exists } T' \in \Tb' \text{ with } T \subset T'.
  \end{equation*}
\end{defi}

\begin{thm}[Bayer--Brandt\cite{BB97}, Athanasiadis\cite{Ath99}]\label{thm:BB-poset}
  Let $n \ge k + 1 \ge 2$ and let $\Oc(n,k)$ be the set of real (resp. complex) generic arrangements $\Ac^0$ of rank $k$ consisting of $n$ hyperplanes such that $L(\Bc(n,k,\Ac^0)) \cong P(n,k)$. Then $\Oc(n,k)$ is a non-empty Zariski open set in the space of real (resp. complex) generic arrangements of rank $k$ consisting of $n$ hyperplanes.

  Furthermore, for $\Ac^0 \in \Oc(n,k)$, the isomorphism between $P(n,k)$ and $L(\Bc(n,k,\Ac^0))$ is given by
  \begin{equation*}
    \Tb \longmapsto
    \bigcap_{T \in \Tb}\ D_T.
  \end{equation*}
  The rank function on $P(n, k)$ is
  \begin{equation*}
    \rk(\Tb) = \sum_{T \in \Tb}(\abs{T} - k),
  \end{equation*}
  and the codimension of the intersection corresponding to $\Tb$ is $\rk(\Tb)$.
\end{thm}

Following Bayer and Brandt~\cite{BB97}, we say that the underlying arrangement $\Ac^0$ is \emph{very generic} if the intersection lattice of $\Bc(n,k,\Ac^0)$ is isomorphic to $P(n,k)$; we call it \emph{non-very generic} otherwise.

Since the set of underlying very generic arrangements $\Ac^0$ is a non-empty Zariski open set $\Oc(n,k)$ of the parameter space of generic arrangements, it is connected in the complex topology.
Hence, any two very generic arrangements are lattice-isotopic.
By Randell's lattice-isotopy theorem~\cite{Ran89}, the complements of lattice-isotopic arrangements are diffeomorphic, and hence have the same homotopy type.

\begin{prop}\label{prop:same-homotopy-type}
  The homotopy types of the complements of the Manin--Schechtman arrangements $\Bc(n,k,\Ac^0)$ for very generic arrangements $\Ac^0$ are independent of the choice of a very generic arrangement.
\end{prop}

Therefore, to prove that $\Bc(n,k,\Ac^0)$ for very generic arrangements is not $K(\pi,1)$, it suffices to exhibit a single realization with the desired property.
In particular, the homotopy type of the complexification of the real-coefficient Manin--Schechtman arrangements coincides with that of the complex-coefficient one.
We will use this observation to show that $\Sigma_2(\Bc(5,2)) = \Sigma_3(\Bc(5,2))$ (see \zcref{sec:pfoflem41}).

Until \zcref{sec:non-very}, this paper focuses on real very generic arrangements.
We will henceforth denote the Manin--Schechtman arrangements by $\Bc(n,k) \coloneq \Bc(n,k,\Ac^0)$ for real very generic arrangements $\Ac^0$, and we treat $\Bc(n,k)$ as real arrangements.

\subsection{Localization}
Fix an isomorphism between $P(n,k)$ and $L(\Bc(n,k))$ in \zcref{thm:BB-poset} and let $\Tb(X) \in P(n,k)$ denote the collection corresponding to $X \in L(\Bc(n,k))$ by the isomorphism.
For $T \subset [n]$, we denote by $\Bc(T,k)$ the subarrangement of $\Bc(n,k)$ consisting of the hyperplanes $\big\{ D_I \in \Bc(n,k) \mid I \in \binom{T}{k+1} \big\}$.

The following proposition is a refinement of Koizumi--Numata--Takemura \cite{KNT12}.

\begin{prop}\label{prop:localization-direct-sum}
  Let $X \in L(\Bc(n,k))$.
  Then
  \begin{equation*}
    \Bc(n,k)_X^{\mathrm{ess}} \simeq \bigoplus_{T \in \Tb(X)} \Bc(T,k)^{\mathrm{ess}}.
  \end{equation*}
\end{prop}
\begin{proof}
  We define the normal space of the localization by $X^{\perp} \coloneq \{ \alpha \in \Sb^\ast \mid X \subset \ker \alpha \}$.
  By the definition of elements in $P(n,k)$ (\zcref{thm:BB-poset}), the hyperplanes in the localization $\Bc(n,k)_X$ are precisely those indexed by $(k + 1)$-subsets contained in one of $T \in \Tb(X)$.
  Hence, we have $X^{\perp} = \sum_{T \in \Tb(X)} D_T^\perp$.
  Moreover, for each $T \in \Tb$ we have $\dim D_T^\perp = \abs{T} - k$ while the rank formula in \zcref{thm:BB-poset} gives $\dim X^\perp = \codim(X) = \sum_{T \in \Tb(X)}(\abs{T} - k)$.
  Therefore, we have $\dim X^\perp = \sum_{T \in \Tb(X)} \dim D_T^\perp$ and the sum is direct, i.e., $X^\perp = \bigoplus_{T \in \Tb(X)} D_T^\perp$.
  Since the hyperplanes of $\Bc(n,k)_X$ are exactly those coming from the factors $\Bc(T, k)$, the claimed arrangement-level decomposition follows.
\end{proof}

\begin{prop}\label{prop:mixed-no-top-jump}
  Let $X \in L(\Bc(n,k))$ be an intersection of codimension $r = \sum_{T \in \Tb(X)}(\abs{T} - k)$ and assume that $\abs{\Tb(X)} \ge 2$.
  Then the following holds:
  \begin{equation*}
    \Sigma_{r-1} \bigl( \Bc(n,k)_X \bigr) = \Sigma_r \bigl( \Bc(n,k)_X \bigr).
  \end{equation*}
\end{prop}
\begin{proof}
  By \zcref{prop:localization-direct-sum}, \zcref{lem:sigma-essentialization} and \zcref{lem:sigma-product-decomposition}, we have
  \begin{equation*}
    \Sigma_p \bigl( \Bc(n,k)_X \bigr) = \prod_{T \in \Tb(X)} \Sigma_{\min(p, \abs{T} - k)} \bigl( \Bc(T,k) \bigr)
  \end{equation*}
  for every $p$.
  Since $\abs{\Tb(X)} \ge 2$ and $\abs{T} - k \ge 1$ for each $T \in \Tb(X)$, we have $\abs{T} - k < \sum_{T' \in \Tb(X)}(\abs{T'} - k) = r$ for every $T \in \Tb(X)$.
  Therefore, for any $T \in \Tb(X)$, we have $\min(r - 1, \abs{T} - k) = \min(r,\abs{T} - k) = \abs{T} - k$.
  Hence, $\Sigma_{r-1} \bigl( \Bc(n,k)_X \bigr) = \Sigma_r \bigl( \Bc(n,k)_X \bigr)$ holds.
\end{proof}
By \zcref{prop:mixed-no-top-jump}, the proper inclusion $\Sigma_r \bigl( \Bc(n,k) \bigr) \supsetneq \Sigma_{r+1} \bigl( \Bc(n,k) \bigr)$ can occur only when $\abs{\Tb(X)} = 1$.
Combining this with \zcref{lem:extend}, we have the following corollary.
\begin{cor}\label{cor:reduction-to-single-set}
  The proper inclusion $\Sigma_r \bigl( \Bc(n,k) \bigr) \supsetneq \Sigma_{r+1} \bigl( \Bc(n,k) \bigr)$ holds if and only if there exists $T \in \binom{[n]}{k+r+1}$ such that $\Sigma_{r} \bigl( \Bc(T,k) \bigr) \supsetneq \Sigma_{r+1} \bigl( \Bc(T,k) \bigr)$.
\end{cor}

\section{Main result}\label{sec:main_result}
In this section, we prove the following main theorem.
\begin{thm}\label{thm:sigma-jump}
  The following holds:
  \begin{itemize}
    \item $\Sigma_2(\Bc(n,2)) = \Sigma_3(\Bc(n,2))$;
    \item if $k \ge 3$, then there exists a Manin--Schechtman arrangement $\Bc(n,k)$ such that $\Sigma_2(\Bc(n,k)) \supsetneq \Sigma_3(\Bc(n,k))$;
    \item if $n - k \ge 3$, $k \ge 2$, and $3\le p\le n-k-1$, then there exists a Manin--Schechtman arrangement $\Bc(n,k)$ such that $\Sigma_p(\Bc(n,k)) \supsetneq \Sigma_{p+1}(\Bc(n,k))$.
  \end{itemize}
\end{thm}
As an immediate consequence of \zcref{thm:Yoshinaga}, we obtain:
\begin{cor}\label{cor:non-K(pi1)}
  Suppose that $n - k \ge 3$.
  If $k \ge 3$, then $\pi_p(M(\Bc(n,k))) \ne 0$ for $2 \le p \le n - k - 1$.
  If $k = 2$, then $\pi_p(M(\Bc(n,2))) \ne 0$ for $3 \le p \le n - 3$.
  In particular, $\Bc(n, k)$ is not $K(\pi,1)$ for $k\neq 1$ and $(n,k) \neq (5,2)$.
\end{cor}
To prove \zcref{thm:sigma-jump}, we need the following three lemmas.
\begin{lem}\label{lem:sigma-52}
  $\Bc(5,2)$ is clean, i.e., $\Sigma_2 \bigl( \Bc(5,2) \bigr) = \Sigma_3 \bigl( \Bc(5,2) \bigr)$.
\end{lem}
\begin{lem}\label{lem:non-very-3}
  If $n - k = 3$ and $k \ge 3$, then there is a generic but non-very generic arrangement $\Ac^0$ of rank $k$ with $n$ hyperplanes, such that there are three distinct hyperplanes $D_{I_1}, D_{I_2}, D_{I_3} \in \Bc(n,k,\Ac^0)$ such that
  \begin{equation*}
    \codim \left( D_{I_1} \cap D_{I_2} \cap D_{I_3} \right) = 2 .
  \end{equation*}
  Moreover, the following holds: $\Bc(n,k,\Ac^0)_{D_{I_1} \cap D_{I_2} \cap D_{I_3}}=\{D_{I_1},D_{I_2},D_{I_3}\}$.
\end{lem}
\begin{lem}\label{lem:non-very-2}
  If $r = n - k \ge 4$ and $k \ge 2$, then there is a generic but non-very generic arrangement $\Ac^0$ of rank $k$ with $n$ hyperplanes, such that there are $r$ distinct hyperplanes $D_{I_1}, \dots, D_{I_r} \in \Bc(n,k,\Ac^0)$ such that
  \begin{itemize}
    \item $\codim \bigcap_{j \in J} D_{I_j} = \abs{J}$ for any $J \subsetneq [r]$;
    \item $\codim \bigcap_{j=1}^r D_{I_j} = r - 1$;
    \item $\Bc(n,k,\Ac^0)_{\bigcap_{j=1}^r D_{I_j}}=\{D_{I_1},\dots,D_{I_r}\}$.
  \end{itemize}
\end{lem}
The proof of \zcref{lem:sigma-52} is given in \zcref{sec:pfoflem41}, and the proofs of \zcref{lem:non-very-3} and \zcref{lem:non-very-2} are given in \zcref{sec:non-very}.
\begin{proof}[Proof of \zcref{thm:sigma-jump}]
  By \zcref{cor:reduction-to-single-set}, it suffices to show
  \begin{align}
     &                                 & \Sigma_2 \bigl( \Bc(5,2) \bigr)       & = \Sigma_3 \bigl( \Bc(5,2) \bigr),           & \label{eq:non-sigma-jump-2}                 \\
     & \text{and for some }\Bc(k+r,k), & \nonumber                                                                                                                          \\
     &                                 & \Sigma_2 \bigl( \Bc(k+r,k) \bigr)     & \supsetneq \Sigma_3 \bigl( \Bc(k+r,k) \bigr) & (r=3, k \ge 3), \label{eq:sigma-jump-3}     \\
     &                                 & \Sigma_{r-1} \bigl( \Bc(k+r,k) \bigr) & \supsetneq \Sigma_r \bigl( \Bc(k+r,k) \bigr) & (r \ge 4, k \ge 2). \label{eq:sigma-jump-2}
  \end{align}
  Equation~\eqref{eq:non-sigma-jump-2} follows immediately from \zcref{lem:sigma-52}.

  Since the set of very generic arrangements is Zariski open, a small perturbation of a non-very generic arrangement yields a very generic arrangement.
  In particular, since Manin--Schechtman arrangements vary continuously with the underlying arrangement, $\Bc(n,k)$ can be obtained by perturbing a Manin--Schechtman arrangement associated with a non-very generic arrangement.
  By \zcref{lem:non-very-3}, there exists $\Bc(n,k,\Ac^0)$ associated with a non-very generic arrangement, which has $X \in L_2(\Bc(n,k,\Ac^0))$ with $\abs{\Bc(n,k,\Ac^0)_X} = 3$.
  Furthermore, there are no intersections of rank two with multiplicity three in any $\Bc(n,k)$ by \zcref{thm:BB-poset}.
  Thus, the assumption of \zcref{cor:perturbation-method} holds, and this implies \eqref{eq:sigma-jump-3}.

  Similarly, \eqref{eq:sigma-jump-2} follows from \zcref{cor:perturbation-method}, \zcref{thm:BB-poset}, and \zcref{lem:non-very-2}.
\end{proof}

\section{Cleanliness of $\Bc(5,2)$}\label{sec:pfoflem41}
As discussed in \cite{DSS26}, every generic arrangement of five lines in $\Rb^2$ is very generic.
Since there is exactly one reorientation class of uniform rank three oriented matroids on five elements, the oriented matroid of $\Bc(5,2)$, which is the arrangement spanned by these five points, is also unique up to reorientation.
\zcref{fig:MS52} shows the projectivized arrangement $\Bc(5,2)$, where each line is labeled by the corresponding triple $I \in \binom{[5]}{3}$.
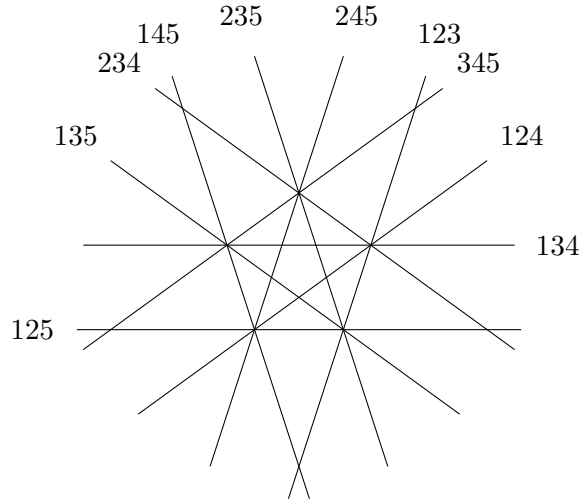
\begin{figure}[h]
  \centering
  \begin{tikzpicture}[scale=1]
    \coordinate (0) at (18:1);
    \coordinate (1) at (90:1);
    \coordinate (2) at (162:1);
    \coordinate (3) at (234:1);
    \coordinate (4) at (306:1);
    \draw ($(1) !3! (0)$) -- ($(0) !3! (1)$) node[pos=1.1] {$234$};
    \draw ($(1) !3! (2)$) -- ($(2) !3! (1)$) node[pos=1.1] {$345$};
    \draw ($(2) !3! (3)$) -- ($(3) !3! (2)$) node[pos=1.1] {$145$};
    \draw ($(3) !3! (4)$) -- ($(4) !3! (3)$) node[pos=1.1] {$125$};
    \draw ($(0) !3! (4)$) -- ($(4) !3! (0)$) node[pos=1.1] {$123$};
    \draw ($(0) !2! (2)$) -- ($(2) !2! (0)$) node[pos=1.1] {$134$};
    \draw ($(0) !2! (3)$) -- ($(3) !2! (0)$) node[pos=1.1] {$124$};
    \draw ($(1) !2! (3)$) -- ($(3) !2! (1)$) node[pos=1.1] {$245$};
    \draw ($(1) !2! (4)$) -- ($(4) !2! (1)$) node[pos=1.1] {$235$};
    \draw ($(2) !2! (4)$) -- ($(4) !2! (2)$) node[pos=1.1] {$135$};
  \end{tikzpicture}
  \caption{The projectivization of $\Bc(5,2)$.}
  \label{fig:MS52}
\end{figure}

For $\sigma \in \Sfrk_5$ and $I \subset [5]$, write $\sigma(I) = \{ \sigma(i) \mid i \in I \}$.
This induces an action on $\binom{[5]}{3}$.
Throughout this section, we write $\Sigma_2 = \Sigma_2(\Bc(5, 2))$, $\Sigma_3 = \Sigma_3(\Bc(5, 2))$, and $\varepsilon_I = \varepsilon_{D_I} \in \{ \pm \}$ for $I \in \binom{[5]}{3}$.

We now classify the possible minimal supports of Gordan certificates arising from sign vectors in $\Sigma_2 \setminus \Sigma_3$ for $\Bc(5,2)$.
Since $\Bc(5,2)$ has rank three, every such certificate may be reduced to one supported on four hyperplanes.
\begin{prop}\label{prop:52-1}
  Let $\varepsilon \in \{ \pm \}^{\binom{[5]}{3}}$.
  If $\varepsilon \in \Sigma_2 \setminus \Sigma_3$, then there exists a subset $\Ib \subset \binom{[5]}{3}$ with $\abs{\Ib} = 4$ and positive coefficients $\lambda_I > 0$ for $I \in \Ib$ such that
  \begin{equation*}
    \sum_{I\in\Ib}\lambda_I \varepsilon_I \alpha_I = 0.
  \end{equation*}
  Moreover $\Ib$ belongs to one of the following $\mathfrak{S}_5$-orbits:
  \begin{equation*}
    \Sfrk_5 \cdot \{123,124,135,245\}, \quad \Sfrk_5 \cdot \{123,124,125,345\}, \quad \Sfrk_5 \cdot \{123,124,135,145\} .
  \end{equation*}
\end{prop}
\begin{proof}
  Suppose that $\varepsilon \in \Sigma_2 \setminus \Sigma_3$.
  Since $\varepsilon \notin \Sigma_3$, there exists $X \in L(\Bc(5,2))$ of codimension three such that $\varepsilon$ is inconsistent at $X$.
  By \zcref{thm:gordan}, there exists a nonzero vector $(\lambda_I)_{D_I \in \Bc(5,2)_X}$ with $\lambda_I \ge 0$ such that
  \begin{equation*}
    \sum_{D_I \in \Bc(5,2)_X} \lambda_I \varepsilon_I \alpha_I = 0 .
  \end{equation*}

  Since $\Bc(5,2)_X$ has rank at most three, by Carath\'eodory's theorem for convex cones \cite[Proposition 1.15]{Zie95}, we may choose such a relation with support $\Ib = \{ I \in \binom{[5]}{3} \mid \lambda_I \ne 0 \}$ of cardinality at most four.
  After removing zero coefficients, we may assume that $\lambda_I > 0$ for every $I \in \Ib$.

  If $\abs{\Ib} \le 2$, this relation is supported in a localization of codimension at most two, contradicting $\varepsilon \in \Sigma_2$.
  Thus $\abs{\Ib} = 3$ or $\abs{\Ib} = 4$.

  First suppose that $\abs{\Ib} = 3$, say $\Ib = \{ I_1, I_2, I_3 \}$.
  Set $T = \bigcup_{j=1}^3 I_j$.
  If $\abs{T} = 5$, then some element $p \in T$ belongs to exactly one of $I_1, I_2, I_3$, since otherwise the total multiplicity would be at least $10$, whereas the total multiplicity is $9 = 3 \times 3$.
  Looking at the $p$-th coordinate of $\sum_{j=1}^3 \lambda_{I_j} \varepsilon_{I_j} \alpha_{I_j} = 0$ gives a contradiction.
  Hence $\abs{T} \le 4$, and the relation is supported in a localization of codimension at most two, again contradicting $\varepsilon \in \Sigma_2$.

  Thus, $\abs{\Ib} = 4$.
  Write $\Ib = \{ I_1, I_2, I_3, I_4 \}$, $T = \bigcup_{j=1}^4 I_j$.
  If $\abs{T} \le 4$, then the relation is supported in a localization of codimension at most two, contradicting $\varepsilon \in \Sigma_2$, hence $\abs{T} = 5$.

  If some $p \in T$ belongs to exactly one of the four triples, then the $p$-th coordinate of $\sum_{j=1}^4 \lambda_{I_j} \varepsilon_{I_j} \alpha_{I_j} = 0$ has only one nonzero term, a contradiction.
  Therefore, every $p \in T$ belongs to at least two of the four triples.

  For $p \in T$, define the multiplicity
  \begin{equation*}
    m(p) = \bigl| \{ j \in \{ 1, 2, 3, 4 \} \mid p \in I_j \} \bigr| .
  \end{equation*}
  Since $\sum_{p \in T} m(p) = 12$ and $m(p) \ge 2$ for every $p \in T$, the multiset $\{ m(p) \mid p \in T \}$ is either $\{ 3, 3, 2, 2, 2 \}$ or $\{ 4, 2, 2, 2, 2 \}$.

  Suppose first that $\{ m(p) \mid p \in T \} = \{ 3, 3, 2, 2, 2 \}$.
  For each $I \in \Ib$, consider its complement $[5] \setminus I$, and define $E_\Ib = \{ [5] \setminus I \mid I \in \Ib \}$.
  Then $E_\Ib \subset \binom{[5]}{2}$ can be regarded as the edge set of a graph on the vertex set $[5]$, where the degree of a vertex $p$ is $4 - m(p)$.
  Hence the degree sequence is $(2, 2, 2, 1, 1)$.
  Up to relabeling, there are exactly two such graphs:
  the path graph $\{ \{ p_1, p_2 \}, \{ p_2, p_3 \}, \{ p_3, p_4 \}, \{ p_4, p_5 \} \}$, and the disjoint union $\{ 12, 34, 35, 45 \} = K_2 \sqcup K_3$ of complete graphs.
  Passing to complements gives $\Ib \in \Sfrk_5 \cdot \{ 123, 124, 135, 245 \}$ or $\Ib \in \Sfrk_5 \cdot \{ 123, 124, 125, 345 \}$.

  Now suppose that $\{ m(p) \mid p \in T \} = \{ 4, 2, 2, 2, 2 \}$.
  After relabeling, we may assume that $m(1) = 4$ and each $I \in \Ib$ is of the form $1ab$ with $a, b \in \{ 2, 3, 4, 5 \}$.
  Removing 1, we obtain $E = \{ I \setminus \{ 1 \} \mid I \in \Ib \} \subset \binom{\{2,3,4,5\}}{2}$.
  Then $E$ can be regarded as the edge set of a graph on the vertex set $\{ 2, 3, 4, 5 \}$.
  Since each of $2, 3, 4, 5$ occurs exactly twice in the members of $\Ib$, every vertex of $E$ has degree two, that is, the graph is a 4-cycle.
  Up to relabeling, we have $E = \{ 23, 24, 35, 45 \}$.
  Considering the complements gives $\Ib \in \Sfrk_5 \cdot \{ 123, 124, 135, 145 \}$, which completes the proof.
\end{proof}

In the remainder of this section, we fix a concrete realization of $\Bc(5, 2)$.
Choose pairwise distinct real numbers $s_1 < \dots < s_5$, and set $\alpha_i = (1, s_i) \in (\Rb^2)^{\ast}$ for $i = 1, \dots, 5$.
Let $\Ac^0$ be the arrangement in $\Rb^2$ with normal vector $\alpha_i$, and let $\Bc(5,2)$ be its associated Manin--Schechtman arrangement.

\begin{prop}\label{prop:52-2}
  Let $\Ib \in \Sfrk_5 \cdot \{ 123, 124, 135, 245 \}$.
  Assume that there exist positive coefficients $\lambda_I$, $I \in \Ib$ and a sign vector $\varepsilon^{(0)} \in \{ \pm \}^\Ib$ such that $\sum_{I \in \Ib} \lambda_I \varepsilon^{(0)}_I \alpha_I = 0$.
  Then every extension $\varepsilon \in \{ \pm \}^{\binom{[5]}{3}}$ with $\varepsilon\vert_\Ib = \varepsilon^{(0)}$ satisfies $\varepsilon \notin \Sigma_2$.
\end{prop}
\begin{proof}
  Let $\Ib_0 = \{ 123, 124, 135, 245 \}$ and choose $\sigma \in \Sfrk_5$ such that $\Ib = \sigma \cdot \Ib_0$.
  A direct computation shows that
  \begin{equation*}
    \lambda_{123}\alpha^\sigma_{123} - \lambda_{124}\alpha^\sigma_{124} + \lambda_{135}\alpha^\sigma_{135} - \lambda_{245}\alpha^\sigma_{245} = 0 ,
  \end{equation*}
  where $\lambda_{123} = (s_{\sigma(1)} - s_{\sigma(5)})(s_{\sigma(2)} - s_{\sigma(4)})$, $\lambda_{124} = (s_{\sigma(1)} - s_{\sigma(3)})(s_{\sigma(2)}-s_{\sigma(5)})$, $\lambda_{135} = (s_{\sigma(1)} - s_{\sigma(2)})(s_{\sigma(2)} - s_{\sigma(4)})$ and $\lambda_{245} = (s_{\sigma(1)} - s_{\sigma(2)})(s_{\sigma(1)} - s_{\sigma(3)})$.

  Define a sign vector $ \eta \in \{ \pm \}^{\Ib_0} $ by
  \begin{equation}\label{eq:52-2-sign}
    \eta_{123} = \sgn(\lambda_{123}), \eta_{124} = -\sgn(\lambda_{124}), \eta_{135} = \sgn(\lambda_{135}), \eta_{245} = -\sgn(\lambda_{245}) .
  \end{equation}
  Then we have a Gordan certificate $\sum_{I \in \Ib_0} \abs{\lambda_I} \eta_I \alpha^\sigma_I = 0$.
  Since any three of these four vectors are linearly independent, this sign vector $\eta$ is uniquely determined up to overall sign.
  Thus, after replacing the extension by its negative if necessary, we may assume $ \varepsilon_I = \eta_I $ for $ I \in \Ib_0 $.

  Suppose, for contradiction, that some extension $\varepsilon \in \{\pm\}^{\binom{[5]}{3}}$ with $\varepsilon\vert_{\Ib_0} = \eta$ belongs to $\Sigma_2$.

  Consider the localization corresponding to $\sigma \cdot \{ 1, 2, 4, 5 \}$.
  We have
  \begin{equation}\label{eq:a125}
    \alpha^\sigma_{125} = \frac{s_{\sigma(2)}-s_{\sigma(5)}}{s_{\sigma(2)}-s_{\sigma(4)}}\alpha^\sigma_{124} + \frac{s_{\sigma(1)}-s_{\sigma(2)}}{s_{\sigma(2)}-s_{\sigma(4)}}\alpha^\sigma_{245} .
  \end{equation}
  By \eqref{eq:52-2-sign}, the two coefficients in \eqref{eq:a125} have the same sign after multiplication by $\varepsilon_{124}$ and $\varepsilon_{245}$.
  Since $\varepsilon \in \Sigma_2$, consistency at the localization corresponding to $\sigma \cdot \{ 1, 2, 4, 5 \}$ forces
  \begin{equation}\label{eq:e125}
    \varepsilon_{125} = - \sgn\left( \frac{s_{\sigma(1)} - s_{\sigma(3)}}{s_{\sigma(2)} - s_{\sigma(4)}} \right) .
  \end{equation}

  Next, consider the localization at $\sigma \cdot \{ 1, 2, 3, 5 \}$.
  We have
  \begin{equation}\label{eq:a135}
    \alpha^\sigma_{135} = \frac{s_{\sigma(5)} - s_{\sigma(1)}}{s_{\sigma(1)} - s_{\sigma(2)}}\alpha^\sigma_{123} + \frac{s_{\sigma(1)} - s_{\sigma(3)}}{s_{\sigma(1)} - s_{\sigma(2)}}\alpha^\sigma_{125} .
  \end{equation}
  Using \eqref{eq:52-2-sign} together with \eqref{eq:e125}, the two coefficients in \eqref{eq:a135} again have the same sign after multiplication by $\varepsilon_{123}$ and $\varepsilon_{125}$.
  Hence consistency at the localization corresponding to $\sigma \cdot \{ 1, 2, 3, 5 \}$ forces
  \begin{equation}
    \varepsilon_{135} = -\sgn\left(\frac{s_{\sigma(2)} - s_{\sigma(4)}}{s_{\sigma(1)} - s_{\sigma(2)}}\right)
  \end{equation}
  which contradicts the sign fixed in \eqref{eq:52-2-sign}.
\end{proof}

\begin{prop}\label{prop:52-3}
  Let $\Ib \in \Sfrk_5 \cdot \{ 123, 124, 125, 345 \}$.
  Assume that there exist positive coefficients $\lambda_I$, $I \in \Ib$ and a sign vector $\varepsilon^{(0)} \in \{ \pm \}^\Ib$ such that $\sum_{I \in \Ib} \lambda_I \varepsilon^{(0)}_I \alpha_I = 0$.
  Then every extension $\varepsilon \in \{ \pm \}^{\binom{[5]}{3}}$ with $\varepsilon\vert_\Ib = \varepsilon^{(0)}$ satisfies $\varepsilon \notin \Sigma_2$.
\end{prop}
\begin{proof}
  The proof is parallel to that of \zcref{prop:52-2}.
  Let $\Ib_0 = \{ 123, 124, 125, 345 \}$ and choose $\sigma \in \Sfrk_5$ such that $\Ib = \sigma \cdot \Ib_0$.
  Since the stabilizer of $\Ib_0$ contains $\Sfrk_{\{ 1, 2 \}} \times \Sfrk_{\{ 3, 4, 5 \}}$, we may choose $\sigma$ so that
  \begin{equation*}
    s_{\sigma(1)} < s_{\sigma(2)}, \quad s_{\sigma(3)} < s_{\sigma(4)} < s_{\sigma(5)} .
  \end{equation*}
  For $ i < j < l $, set $\alpha^\sigma_{ijl} \coloneq \alpha_{\sigma(i) \sigma(j) \sigma(l)}$.
  A direct computation gives
  \begin{equation*}
    (s_{\sigma(4)} - s_{\sigma(5)}) \alpha^\sigma_{123} - (s_{\sigma(3)} - s_{\sigma(5)}) \alpha^\sigma_{124} + (s_{\sigma(3)} - s_{\sigma(4)}) \alpha^\sigma_{125} - (s_{\sigma(1)} - s_{\sigma(2)}) \alpha^\sigma_{345} = 0 .
  \end{equation*}
  Hence, after replacing the extension by its negative if necessary, we may assume
  \begin{equation*}
    (\varepsilon_{123}, \varepsilon_{124}, \varepsilon_{125}, \varepsilon_{345}) = (-, +, -, +) ,
  \end{equation*}
  where $\varepsilon_{ijk}$ denotes the sign assigned to $\alpha^\sigma_{ijk}$.

  We now distinguish four cases according to the position of $s_{\sigma(1)}$
  among
  $ s_{\sigma(3)} < s_{\sigma(4)} < s_{\sigma(5)} $.

  \noindent \textbf{Case 1:}
  $ s_{\sigma(1)} < s_{\sigma(3)} < s_{\sigma(4)} < s_{\sigma(5)} $ .

  In the localization corresponding to $\sigma \cdot \{ 1, 2, 3, 4 \} $, we have
  \begin{equation}\label{eq:a134}
    \alpha^\sigma_{134} = \frac{s_{\sigma(4)} - s_{\sigma(1)}}{s_{\sigma(1)} - s_{\sigma(2)}} \alpha^\sigma_{123} + \frac{s_{\sigma(1)} - s_{\sigma(3)}}{s_{\sigma(1)} - s_{\sigma(2)}} \alpha^\sigma_{124} .
  \end{equation}
  The two coefficients in \eqref{eq:a134} have the same sign after multiplication by $\varepsilon_{123}$ and $\varepsilon_{124}$.
  Since $\varepsilon \in \Sigma_2$, consistency at this localization forces $\varepsilon_{134} = +$.
  Similarly, at the localization $\sigma \cdot \{ 1, 2, 4, 5 \} $, we have
  \begin{equation}\label{eq:a145}
    \alpha^\sigma_{145} = \frac{s_{\sigma(5)} - s_{\sigma(1)}}{s_{\sigma(1)} - s_{\sigma(2)}} \alpha^\sigma_{124} + \frac{s_{\sigma(1)} - s_{\sigma(4)}}{s_{\sigma(1)} - s_{\sigma(2)}} \alpha^\sigma_{125} ,
  \end{equation}
  and consistency forces $\varepsilon_{145} = -$.

  Now consider the localization at $\sigma \cdot \{ 1, 3, 4, 5 \}$.
  We have
  \begin{equation}\label{eq:a134_2}
    \alpha^\sigma_{134} = \frac{s_{\sigma(3)} - s_{\sigma(4)}}{s_{\sigma(4)} - s_{\sigma(5)}} \alpha^\sigma_{145} + \frac{s_{\sigma(4)} - s_{\sigma(1)}}{s_{\sigma(4)} - s_{\sigma(5)}} \alpha^\sigma_{345} .
  \end{equation}
  The two coefficients have the same sign after multiplication by $\varepsilon_{145}$ and $\varepsilon_{345}$.
  Hence consistency forces $\varepsilon_{134} = -$, a contradiction.

  \noindent\textbf{Case 2:}
  $s_{\sigma(3)} < s_{\sigma(1)} < s_{\sigma(4)} < s_{\sigma(5)}$.

  At the localization $\sigma \cdot \{ 1, 2, 3, 5 \} $, using the relation \eqref{eq:a135}, consistency forces $\varepsilon_{135} = +$.
  Also, at the localization $\sigma \cdot \{ 1, 2, 4, 5 \}$, from the relation \eqref{eq:a145} we have $\varepsilon_{145} = -$.
  On the other hand, at the localization $\sigma \cdot \{ 1, 3, 4, 5 \}$, we have
  \begin{equation*}
    \alpha^\sigma_{135} = \frac{s_{\sigma(3)} - s_{\sigma(5)}}{s_{\sigma(4)} - s_{\sigma(5)}} \alpha^\sigma_{145} + \frac{s_{\sigma(5)} - s_{\sigma(1)}}{s_{\sigma(4)} - s_{\sigma(5)}} \alpha^\sigma_{345} .
  \end{equation*}
  Under the present order, the two coefficients have the same sign after multiplication by $\varepsilon_{145}$ and $\varepsilon_{345}$, so consistency forces $\varepsilon_{135} = -$, a contradiction.

  \noindent\textbf{Case 3:} $ s_{\sigma(3)} < s_{\sigma(4)} < s_{\sigma(1)} < s_{\sigma(5)} $.

  At the localization $\sigma \cdot \{ 1, 2, 3, 4 \}$, the relation \eqref{eq:a134} forces $\varepsilon_{134} = -$, and at the localization $\sigma \cdot \{ 1, 2, 3, 5 \}$, the relation \eqref{eq:a135} forces $\varepsilon_{135} = +$.
  On the other hand, at the localization $\sigma \cdot \{ 1, 3, 4, 5 \}$, we have
  \begin{equation*}
    \alpha^\sigma_{134} = \frac{s_{\sigma(3)} - s_{\sigma(4)}}{s_{\sigma(3)} - s_{\sigma(5)}} \alpha^\sigma_{135} + \frac{s_{\sigma(3)} - s_{\sigma(1)}}{s_{\sigma(3)} - s_{\sigma(5)}} \alpha^\sigma_{345}.
  \end{equation*}
  which forces $\varepsilon_{134} = +$, a contradiction.

  \noindent
  \textbf{Case 4:} $s_{\sigma(3)} < s_{\sigma(4)} < s_{\sigma(5)} < s_{\sigma(1)} $.

  At the localization $\sigma \cdot \{ 1, 2, 3, 4 \}$, the relation \eqref{eq:a134} forces $\varepsilon_{134} = -$ and at the localization $\sigma \cdot \{ 1, 2, 4, 5 \}$, the relation \eqref{eq:a145} forces $\varepsilon_{145} = +$.
  On the other hand, at the localization $\sigma \cdot \{ 1, 3, 4, 5 \}$, using the relation \eqref{eq:a134_2}, consistency forces $\varepsilon_{134} = +$, a contradiction.

  In all cases, we obtain a contradiction.
  Hence, no extension belongs to $\Sigma_2$.
\end{proof}

\begin{prop}\label{prop:52-4}
  Let $\Ib \in \Sfrk_5 \cdot \{  123, 124, 135, 145 \}$.
  Assume that there exist positive coefficients $\lambda_I$, $I \in \Ib$ and a sign vector $\varepsilon^{(0)} \in \{ \pm \}^\Ib$ such that $\sum_{I \in \Ib} \lambda_I \varepsilon^{(0)}_I \alpha_I = 0$.
  Then every extension $\varepsilon \in \{ \pm \}^{\binom{[5]}{3}}$ with $\varepsilon\vert_\Ib = \varepsilon^{(0)}$ satisfies $\varepsilon \notin \Sigma_2$.
\end{prop}
\begin{proof}
  The proof is again parallel to that of \zcref{prop:52-2}.
  Let $\Ib_0 = \{ 123, 124, 135, 145 \}$ and choose $\sigma \in \Sfrk_5$ such that $\Ib = \sigma \cdot \Ib_0$.
  For $i < j < l$, set $\alpha^\sigma_{ijl} \coloneq \alpha_{\sigma(i)\sigma(j)\sigma(l)}$.

  A direct computation gives
  \begin{equation*}
    \lambda_{123} \alpha^\sigma_{123} - \lambda_{124} \alpha^\sigma_{124} + \lambda_{135} \alpha^\sigma_{135} - \lambda_{145} \alpha^\sigma_{145} = 0 ,
  \end{equation*}
  where $\lambda_{123} =(s_{\sigma(1)} - s_{\sigma(4)})(s_{\sigma(1)} - s_{\sigma(5)})$, $\lambda_{124} =(s_{\sigma(1)} - s_{\sigma(3)})(s_{\sigma(1)} - s_{\sigma(5)})$, $\lambda_{135} =(s_{\sigma(1)} - s_{\sigma(2)})(s_{\sigma(1)} - s_{\sigma(4)})$ and $\lambda_{145} =(s_{\sigma(1)} - s_{\sigma(2)})(s_{\sigma(1)} - s_{\sigma(3)})$.

  Define a sign vector $\eta \in \{ \pm \}^{\Ib_0}$ by
  \begin{equation}\label{eq:52-4-sign}
    \eta_{123} = \sgn(\lambda_{123}),\quad \eta_{124} = - \sgn(\lambda_{124}),\quad \eta_{135} = \sgn(\lambda_{135}),\quad \eta_{145} = -\sgn(\lambda_{145}) .
  \end{equation}
  Then we have a Gordan certificate $\sum_{I \in \Ib_0} \abs{\lambda_I} \eta_I \alpha^\sigma_I = 0$.
  Since any three of these four vectors are linearly independent, this sign vector is uniquely determined up to overall sign.
  Thus, after replacing the extension by its negative if necessary, we may assume $\varepsilon_I = \eta_I$, $I \in \Ib_0$.

  Suppose, for contradiction, that some extension $\varepsilon \in \{ \pm \}^{\binom{[5]}{3}}$ with $\varepsilon\vert_{\Ib_0} = \eta$ belongs to $\Sigma_2$.
  First, consider the localization corresponding to $\sigma \cdot \{ 1, 2, 3, 4 \}$.
  Using the relation \eqref{eq:a134} together with \eqref{eq:52-4-sign}, the two coefficients have the same sign after multiplication by $\varepsilon_{123}$ and $\varepsilon_{124}$.
  Hence consistency forces $\varepsilon_{134} = - \sgn \left( \frac{s_{\sigma(1)} - s_{\sigma(5)}}{s_{\sigma(1)} - s_{\sigma(2)}} \right)$.

  Next, consider the localization at $\sigma \cdot \{ 1, 3, 4, 5 \}$.
  We also have
  \begin{equation*}
    \alpha^\sigma_{134} = \frac{s_{\sigma(1)} - s_{\sigma(4)}}{s_{\sigma(1)} - s_{\sigma(5)}} \alpha^\sigma_{135} + \frac{s_{\sigma(3)} - s_{\sigma(1)}}{s_{\sigma(1)} - s_{\sigma(5)}} \alpha^\sigma_{145} .
  \end{equation*}
  Again, using the sign fixed in \eqref{eq:52-4-sign}, the two signs agree, and consistency forces $\varepsilon_{134} = \sgn \left( \frac{s_{\sigma(1)} - s_{\sigma(2)}}{s_{\sigma(1)} - s_{\sigma(5)}} \right)$, a contradiction.
  Hence, no extension belongs to $\Sigma_2$.
\end{proof}

\begin{rem}
  \zcref{prop:52-2} and \zcref{prop:52-4} can also be proved by an argument based on the strong elimination property for circuits, as used in the proof of Theorem~3.2 of \cite{DDBP25}.
  For instance, we have
  \begin{equation*}
    \{ 123, 124, 135, 245 \} = \bigl( \{ 123, 125, 135 \} \cup \{ 124, 125, 245 \} \bigr) \setminus \{ 125 \} .
  \end{equation*}
  Thus, the circuit of size four is obtained by eliminating $125$ from the two circuits of size three
  \begin{equation*}
    C_1 = \{ 123, 125, 135 \}, \quad C_2 = \{ 124, 125, 245 \} .
  \end{equation*}
  Fix a sign on $\{ 123, 124, 135, 245 \}$, up to overall sign, and let $\varepsilon$ be any extension of this sign to all indices.

  Since $C_1$ and $C_2$ are circuits of size three, consistency determines the sign of $\varepsilon_{125}$ uniquely from the signs on $\{ 123, 135 \}$ and  $\{ 124, 245 \}$, respectively.
  These two induced signs are opposite.
  Hence, $\varepsilon_{125}$ agrees with exactly one of them.
  Consequently, one of the restrictions $\varepsilon\vert_{C_1}$, $\varepsilon\vert_{C_2}$ violates consistency.
  Therefore, one of the corresponding triples of half-spaces has empty intersection, and hence $\varepsilon \notin \Sigma_2$.
  Similarly, $\{ 123, 124, 135, 145 \} = \bigl( \{ 123, 124, 134 \} \cup \{ 134, 135, 145 \} \bigr) \setminus \{ 134 \}$, and the same argument shows that any extension violates consistency on one of the two circuits of size three.

  On the other hand, this argument does not apply to
  Proposition~\ref{prop:52-3}.
  Indeed, the circuit $\{ 123, 124, 125, 345 \}$ does not admit such a decomposition $C'' = (C_1'' \cup C_2'') \setminus \{ h \}$ with $C_1'' \cap C_2'' = \{ h \}$ and $\abs{C_1''}, \abs{C_2''} \ge 3$.
  Hence, it is not chordal in the sense of \cite{DDBP25}, and a separate case analysis is required.
\end{rem}
We now prove \zcref{lem:sigma-52}, as announced in \zcref{sec:main_result}.
\begin{proof}[Proof of \zcref{lem:sigma-52}]
  Since $\Sigma_2\supseteq \Sigma_3$ by definition, it suffices to prove $\Sigma_2\subseteq \Sigma_3$.
  Suppose, for contradiction, that $\varepsilon\in \Sigma_2\setminus \Sigma_3$.
  Then \zcref{prop:52-1} gives a Gordan certificate supported on one of the three $\mathfrak S_5$-orbits $\Sfrk_5 \cdot \{ 123, 124, 135, 245 \}$, $\Sfrk_5 \cdot \{ 123, 124, 125, 345 \}$ and $\Sfrk_5 \cdot \{ 123, 124, 135, 145 \}$.
  By \zcref{prop:52-2}, \zcref{prop:52-3}, and \zcref{prop:52-4}, any extension of the corresponding sign vector fails to lie in $\Sigma_2$.
  This contradicts $\varepsilon \in \Sigma_2$.
  Hence $\Sigma_2 \setminus \Sigma_3 = \emptyset$, and therefore $\Sigma_2 = \Sigma_3$.
\end{proof}

\section{Existence of non-very generic arrangements}\label{sec:non-very}
To prove \zcref{lem:non-very-3} and \ref{lem:non-very-2} in terms of parameter spaces of generic arrangements, we first recall the relevant setup.
Following \cite{Fal94,SSY17}, we realize labelled arrangements by points of a Grassmannian $\mathrm{Gr}_{k}(\Cb^{n})$, where $\mathrm{Gr}_{k}(\Cb^{n})$ consists of all $k$-dimensional subspaces of $\Cb^{n}$.

Let $\Ac=\{H_1,\dots,H_n\}$ be an essential arrangement in $\Cb^{k}$ with normal vectors $\alpha_1,\dots,\alpha_n\in (\Cb^{k})^{\ast}$ and let $W(\Ac)$ be the image of the linear map $(\alpha_1,\dots,\alpha_n)$.
Since $\Ac$ is essential, $W(\Ac)$ is the $k$--dimensional linear subspace of $\Cb^{n}$, i.e., a point in the Grassmannian $\mathrm{Gr}_{k}(\Cb^{n})$.
Different choices of normal vectors give points in the same $(\Cb^{\times})^{n}$-orbit in $\mathrm{Gr}_{k}(\Cb^{n})$.

Conversely, a point $W\in \mathrm{Gr}_{k}(\Cb^{n})$ determines an arrangement inside $W$ by intersecting $W$ with the coordinate hyperplanes of $\Cb^n$.
If $W=W(\Ac)$, this arrangement is linearly equivalent to the original arrangement $\Ac$.

We denote by $\mathrm{Gr}_{k}^{\circ}(\Cb^{n})$ the Zariski open subset of $\mathrm{Gr}_{k}(\Cb^{n})$ corresponding to generic arrangements.
Equivalently, it is the subset on which all Pl\"ucker coordinates $\Delta_{i_1\dots i_k}=\det(\alpha_{i_1}\dots\alpha_{i_k})$, $\{i_1,\dots,i_k\}\in\binom{[n]}{k}$, are nonzero.

In \cite{ABFKST23} and \cite{Sai26}, the moduli of generic arrangements is considered as the quotient $\mathrm{Gr}_{k}^{\circ}(\Cb^{n})/(\Cb^\times)^n$ of $\mathrm{Gr}_{k}^{\circ}(\Cb^{n})$ by the torus action.
In this paper, however, we work on the parameter space $\mathrm{Gr}_{k}^{\circ}(\Cb^{n})$ itself, as we will discuss its relationship with the matroid strata in \zcref{subsec:matroid_strata}.

\subsection{Non-very generic intersection and hypersurfaces in Grassmannian}
For the remainder of this paper, we do not require generic arrangements $\Ac^0$ to be very generic.
\begin{defi}
  Let $X\in L(\Bc(n,k,\Ac^0))$. We call $D_{T}$ a \emph{component} of $X$ if $\abs{T}\ge k+1$, $X\subset D_{T}$, and $X\not\subset D_{T\cup\{i\}}$ for all $i \in [n]\setminus T$. We define a map
  \begin{align*}
    \Tb:L(\Bc(n,k,\Ac^0)) & \to 2^{2^{[n]}}                                                      \\
    X                     & \mapsto \Tb(X)=\{T\in 2^{[n]}\mid D_{T}\text{ is a component of }X\}
  \end{align*}
  and call $\Tb(X)$ the \emph{canonical presentation} of $X$.
  We call $X\in L(\Bc(n,k,\Ac^0))$ \emph{very generic} if $\codim X=\sum_{T\in\Tb(X)}\codim D_{T}$, and \emph{non-very generic} otherwise.
\end{defi}
Note that $\bigcap_{T\in\Tb(X)}D_T=X$.
Also, if $\Ac^0$ is very generic, then the image of the intersection lattice $L(\Bc(n,k,\Ac^0))$ is isomorphic to $P(n,k)$.
\begin{ex}
  [\cite{Fal94}] Let $\Ac^0$ be the arrangement in $\Cb^{3}$ with normal vectors
  \begin{equation*}
    (\alpha_1,\dots,\alpha_6)=
    \begin{pmatrix}
      1 & 0 & 0 & 1 & 2 & 2 \\
      0 & 1 & 0 & 2 & 3 & 2 \\
      0 & 0 & 1 & 2 & 2 & 1
    \end{pmatrix}.
  \end{equation*}
  The following equality holds:
  \begin{equation*}
    \operatorname{rank}
    \begin{pmatrix}
      \alpha_{1245} \\
      \alpha_{1346} \\
      \alpha_{2356}
    \end{pmatrix}
    =\operatorname{rank}
    \begin{pmatrix}
      2  & 2  & 0  & 2  & -2 & 0  \\
      -2 & 0  & 2  & -2 & 0  & 2  \\
      0  & -2 & -2 & 0  & 2  & -2
    \end{pmatrix}=2< \sum_{i=1}^{3}\codim D_{I}=3.
  \end{equation*}
  and $\Tb(D_{1245}\cap D_{1346}\cap D_{2356})=\{1245,1346,2356\}\notin P(6,3)$. Thus, $D_{1245}\cap D_{1346}\cap D_{2356}$ is a non-very generic intersection, and hence $\Ac^0$ is a non-very generic arrangement.
  This example is illustrated to the right of \zcref{fig:non-very-falk}.
  We can see that the triangular chamber in the center has degenerated compared to the very generic case (left of \zcref{fig:non-very-falk}).
  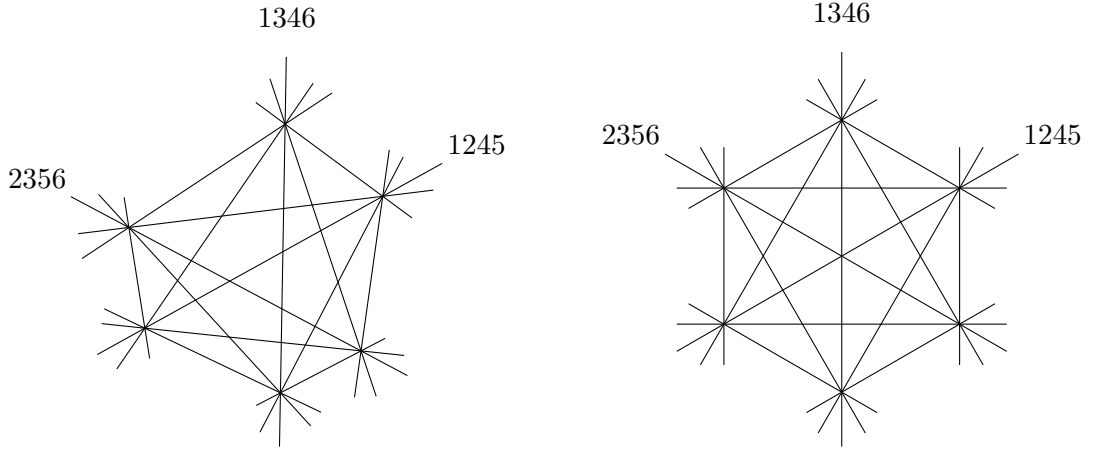
\begin{figure}[htbp]
    \centering
    \begin{tikzpicture}[scale=1.8]
      \coordinate (0) at (27:1);
      \coordinate (1) at (80:1);
      \coordinate (2) at (167:1);
      \coordinate (3) at (211:1);
      \coordinate (4) at (278:1);
      \coordinate (5) at (317:1);
      \draw ($(0) !1.3! (1)$) -- ($(1) !1.3! (0)$);
      \draw ($(1) !1.3! (2)$) -- ($(2) !1.3! (1)$);
      \draw ($(2) !1.3! (3)$) -- ($(3) !1.3! (2)$);
      \draw ($(3) !1.3! (4)$) -- ($(4) !1.3! (3)$);
      \draw ($(4) !1.3! (5)$) -- ($(5) !1.3! (4)$);
      \draw ($(0) !1.3! (5)$) -- ($(5) !1.3! (0)$);
      \draw ($(4) !1.2! (0)$) -- ($(0) !1.2! (4)$);
      \draw ($(0) !1.2! (2)$) -- ($(2) !1.2! (0)$);
      \draw ($(0) !1.2! (3)$) -- ($(3) !1.25! (0)$) node[pos=1.1] {$1245$};
      \draw ($(1) !1.2! (3)$) -- ($(3) !1.2! (1)$);
      \draw ($(1) !1.2! (4)$) -- ($(4) !1.25! (1)$) node[pos=1.1] {$1346$};
      \draw ($(2) !1.2! (4)$) -- ($(4) !1.2! (2)$);
      \draw ($(1) !1.2! (5)$) -- ($(5) !1.2! (1)$);
      \draw ($(2) !1.2! (5)$) -- ($(5) !1.25! (2)$) node[pos=1.1] {$2356$};
      \draw ($(3) !1.2! (5)$) -- ($(5) !1.2! (3)$);
    \end{tikzpicture}
    \hspace{0.05\columnwidth}
    \begin{tikzpicture}[scale=1.8]
      \coordinate (0) at (30:1);
      \coordinate (1) at (90:1);
      \coordinate (2) at (150:1);
      \coordinate (3) at (210:1);
      \coordinate (4) at (270:1);
      \coordinate (5) at (330:1);
      \draw ($(0) !1.3! (1)$) -- ($(1) !1.3! (0)$);
      \draw ($(1) !1.3! (2)$) -- ($(2) !1.3! (1)$);
      \draw ($(2) !1.3! (3)$) -- ($(3) !1.3! (2)$);
      \draw ($(3) !1.3! (4)$) -- ($(4) !1.3! (3)$);
      \draw ($(4) !1.3! (5)$) -- ($(5) !1.3! (4)$);
      \draw ($(0) !1.3! (5)$) -- ($(5) !1.3! (0)$);
      \draw ($(4) !1.2! (0)$) -- ($(0) !1.2! (4)$);
      \draw ($(0) !1.2! (2)$) -- ($(2) !1.2! (0)$);
      \draw ($(0) !1.2! (3)$) -- ($(3) !1.25! (0)$) node[pos=1.1] {$1245$};
      \draw ($(1) !1.2! (3)$) -- ($(3) !1.2! (1)$);
      \draw ($(1) !1.2! (4)$) -- ($(4) !1.25! (1)$) node[pos=1.1] {$1346$};
      \draw ($(2) !1.2! (4)$) -- ($(4) !1.2! (2)$);
      \draw ($(1) !1.2! (5)$) -- ($(5) !1.2! (1)$);
      \draw ($(2) !1.2! (5)$) -- ($(5) !1.25! (2)$) node[pos=1.1] {$2356$};
      \draw ($(3) !1.2! (5)$) -- ($(5) !1.2! (3)$);
    \end{tikzpicture}
    \caption{Projectivizations of $\Bc(6,3,\Ac^0)$, (left: very generic case, right: Falk's non-very generic case)}
    \label{fig:non-very-falk}
  \end{figure}
\end{ex}

To show \zcref{lem:non-very-3} and \ref{lem:non-very-2}, we want a Manin--Schechtman arrangement $\Bc(n,k,\Ac^0)$ of a given family $\Tb$ of sets that has an intersection $X\in L(\Bc(n,k,\Ac^0))$ at a given rank such that the family $\Tb$ of sets is canonical presentation, i.e., $\Tb=\Tb(X)$. Such a set of arrangements can be described algebraically as follows:
\begin{defi}
  Let $\Tb$ be the antichain of $2^{[n]}$ and satisfy the following conditions:
  \begin{enumerate}[label=(Q\arabic*), ref=Q\arabic*]
    \item\label{item:Q1} each element of $\Tb$ has at least $k + 1$ elements,
    \item\label{item:Q2} each $k$-element subset of $[n]$ is contained in at most one member of $\Tb$.
  \end{enumerate}
  For non-negative integer $r$, we define
  \begin{equation*}
    V^{\circ}_{(\Tb,r)}\coloneq\left\{\Ac^0\in\mathrm{Gr}_{k}^{\circ}(\Cb^{n}) \mathrel{}\middle|\mathrel{}\codim_{\Sb(\Ac^0)}\bigcap_{T\in\Tb}D_{T}(\Ac^0)\le r
    \right\}.
  \end{equation*}
\end{defi}
Obviously, if $X\in L_r(\Bc(n,k,\Ac^0))$, then we have $\Ac^0\in V^{\circ}_{(\Tb(X),r)}$.
Since $V^{\circ}_{(\Tb,r)}$ is defined by linear dependencies, it is a subvariety of $\mathrm{Gr}_{k}^{\circ}(\Cb^{n})$.
The defining equation for $(\Tb,r)=(\{1245,1346,2356\},2)$ is given in \cite{SSY17}.
These varieties with different parameters have also been discussed in \cite{SS25,Sai26}.
We give a sufficient condition for $V^{\circ}_{(\Tb,r)}$ to be a hypersurface. This is a generalization of \cite{SSY17}.

\begin{prop}\label{prop:hypersurface}
  Let $\Tb$ be the antichain of $2^{[n]}$ satisfying \eqref{item:Q1} and \eqref{item:Q2}.
  If $\abs{\bigcup_{T\in \Tb}T}-k=\sum_{T\in \Tb}(\abs{T}-k)=r+1$, then the variety $V^{\circ}_{(\Tb,r)}$ is either empty or a hypersurface in $\mathrm{Gr}_{k}^{\circ}(\Cb^{n})$.
\end{prop}
\begin{proof}
  Consider the subspace $\Sb(\Ac^0)^{\perp}_{\Tb}\coloneq \left(\bigcap_{T\in\Tb}D_{T}\right)^{\perp}=\sum_{T\in\Tb}D_{T}^{\perp}$ of $\Sb(\Ac^0)^{\ast}$, which is spanned by
  \begin{equation*}
    \left\{\alpha_{I}\mathrel{}\middle|\mathrel{}I\in \binom{T}{k+1}, T\in\Tb\right\}.
  \end{equation*}
  By definition, $\Ac^0\in V^{\circ}_{(\Tb,r)}$ if and only if $\codim \bigcap_{T\in \Tb}D_{T}\le r$, which is equivalent to $\dim \Sb(\Ac^0)^{\perp}_{\Tb}\le r$.
  Let $A_{\Tb}(\Ac^0)$ be an $(r+1)\times n$ matrix whose rows consist of normal vectors
  \begin{equation*}
    \bigcup_{T\in \Tb}\{\alpha_{I(T,1)},\dots, \alpha_{I(T,\abs{T}-k)}\},
  \end{equation*}
  where ${I(T,j)}\coloneq\{i_{1},\dots,i_{k},i_{k+j}\}\subset [n]$ for $T = \{i_{1}<\dots<i_{\abs{T}}\}$ and $1\le j\le \abs{T}-k$.
  Since the row space of $A_{\Tb}(\Ac^0)$ is $\Sb(\Ac^0)^{\perp}_{\Tb}$, it is enough to show $\operatorname{rank}A_{\Tb}(\Ac^0)\le r$ if and only if the vanishing of the determinant of a single $(r+1)$-minor.
  Note that the columns not indexed by subsets of $\bigcup_{T\in\Tb}T$ are zero vectors.
  Thus, $\operatorname{rank}A_{\Tb}(\Ac^0)\le r$ if and only if all $(r+1)$-minors of $A_{\Tb}(\Ac^0)$ vanish.
  Note that, if $\Ac^0$ is very generic, then the intersection $\bigcap_{T\in\Tb}D_T$ corresponds to $D_{\bigcup_{T\in \Tb}T}$ in the Bayer--Brandt--Athanasiadis description (\zcref{thm:BB-poset}), because $\sum_{T\in\Tb}(\abs{T}-k)=\abs{\bigcup_{T\in \Tb}T}-k$, which implies $\dim \Sb(\Ac^0)^{\perp}_{\Tb}= r+1$.
  Thus, an $(r+1)$-minor of $A_{\Tb}(\Ac^0)$ is not identically zero.

  We show that if one of the $(r+1)$-minors of $A_{\Tb}(\Ac^0)$ whose columns are indexed by subsets of $\bigcup_{T\in \Tb}T$ vanishes, then all $(r+1)$-minors of $A_{\Tb}(\Ac^0)$ vanish.
  Let $S_{1}$ be a subset of $\bigcup_{T\in\Tb}T$ of cardinality $r+1$.
  Let $S_{2}=S_{1}\cup \{i_{2}\}\setminus \{i_{1}\}$ for arbitrary choices of $i_{1}\in S_{1}$ and $i_{2}\in \bigcup_{T\in\Tb}T\setminus S_{1}$, and let $S_0=S_{1}\cup\{i_{2}\}=S_{2}\cup\{i_{1}\}$.
  Then we have $\abs{S_0}=r+2$.

  Suppose that the minor $\det (A_{\Tb}(\Ac^0)|_{S_1})$ of $A_{\Tb}(\Ac^0)$ whose columns are indexed by $S_{1}$ vanishes.
  If we take $K=(\bigcup_{T\in \Tb}T\setminus S_{0})\cup\{i_{1},i_2\}\subset \bigcup_{T\in \Tb}T$, then it holds that $\abs{K}=\abs{\bigcup_{T\in \Tb}T}-r=k+1$ and $\alpha_K\in D^\perp_{\bigcup_{T\in \Tb}T}$.
  Since $\Sb(\Ac^0)^{\perp}_{\Tb}\subset D^\perp_{\bigcup_{T\in \Tb}T}$ and $\dim D^\perp_{\bigcup_{T\in \Tb}T} = r+1$, the rank of an extended matrix $\begin{pmatrix} A_{\Tb}(\Ac^0) \\ \alpha_{K} \end{pmatrix}$ is equal to or less than $r+1$, and hence all $(r+2)$-minors vanish.
  In particular, the minor of $\begin{pmatrix} A_{\Tb}(\Ac^0) \\ \alpha_{K} \end{pmatrix}$ corresponding to the columns indexed by $S_0$ vanishes.
  Since $S_0\cap K=\{i_1,i_2\}$, the $(r+2)$-minor of the extended matrix corresponding to the columns indexed by $S_0$ vanishes, that is $c_{K,i_2}\det (A_{\Tb}(\Ac^0)|_{S_1})\pm c_{K,i_1}\det (A_{\Tb}(\Ac^0)|_{S_2})=0$, where $c_{K,i_1},c_{K,i_2}$ are the $i_{1}$-th and $i_{2}$-th entries of $\alpha_{K}$
  Since $\Ac^0\in \mathrm{Gr}_{k}^{\circ}(\Cb^{n})$, both $c_{K,i_1}$ and $c_{K,i_2}$ are non-zero, and hence we obtain $\det (A_{\Tb}(\Ac^0)|_{S_2})=0$.

  By the arbitrariness of $S_{1},i_{1},i_{2}$ and by repeating this argument, we conclude that all $(r+1)$-minors vanish.
  Therefore $V^\circ_{(\Tb,r)}$ is either empty or a hypersurface in $\mathrm{Gr}_{k}^{\circ}(\Cb^{n})$.
\end{proof}

\subsection{Proof of \zcref{lem:non-very-3} and \zcref{lem:non-very-2}}
The determinant of the matrix $(\alpha_{i_1}\dots\alpha_{i_k})$ will be denoted by $\Delta_{i_1\dots i_k}$.

\begin{proof}[Proof of \zcref{lem:non-very-3}]
  Let $I_1\coloneq [n]\setminus\{1,2\}$, $I_2\coloneq [n]\setminus\{3,4 \}$, $I_3\coloneq [n]\setminus\{5,6\}$, $\Tb_{n}\coloneq\{I_1,I_2,I_3\}$, and $\eta_{n}\coloneq[n]\setminus[6]$.
  Then $\Tb_n$ satisfies \eqref{item:Q1} and \eqref{item:Q2} and we have
  \begin{equation*}
    \begin{pmatrix}
      \alpha_{I_1} \\
      \alpha_{I_2} \\
      \alpha_{I_3}
    \end{pmatrix}=
    \begin{pmatrix}
      0                  & 0                   & \Delta_{456\eta_n} & -\Delta_{356\eta_n} & \Delta_{346\eta_n} & -\Delta_{345\eta_n} & \Delta_{[n]\setminus\{1,2,7\}} & \cdots \\
      \Delta_{256\eta_n} & -\Delta_{156\eta_n} & 0                  & 0                   & \Delta_{126\eta_n} & -\Delta_{125\eta_n} & \Delta_{[n]\setminus\{3,4,7\}} & \cdots \\
      \Delta_{234\eta_n} & -\Delta_{134\eta_n} & \Delta_{124\eta_n} & -\Delta_{123\eta_n} & 0                  & 0                   & \Delta_{[n]\setminus\{5,6,7\}} & \cdots
    \end{pmatrix}.
  \end{equation*}
  Since $\sum_{T\in \Tb_n}(\abs{T}-k)=\abs{\bigcup_{T\in \Tb_n}T}-k=3$, we can apply \zcref{prop:hypersurface}, and we have
  \begin{equation*}
    V^\circ_{(\Tb_{n},2)}=\{\Ac^0\in \mathrm{Gr}_{k}^{\circ}(\Cb^{n})\mid \Delta_{125\eta_n}\Delta_{346\eta_n}-\Delta_{126\eta_n}\Delta_{345\eta_n}=0\}
  \end{equation*}
  by calculation of minors for columns of $4$, $5$, and $6$.

  The defining equation of $V^{\circ}_{(\Tb_n,2)}$ is linear in the column $\alpha_1$ after fixing the other columns.
  Moreover, this nonzero linear form is not proportional to any Plücker coordinate $\Delta_{1J}$ with $J\in\binom{[n]\setminus\{1\}}{k-1}$.
  Hence, its zero set is not contained in the union of the hyperplanes $\{\Delta_{1J}=0\}$.
  Since $\Rb$ is an infinite field, we may choose a real point on this zero set satisfying all genericity conditions.
  Thus, we can take $\Ac^0\in V^{\circ}_{(\Tb_n,2)}\cap \mathrm{Gr}_{k}(\Rb^{n})$ and there are three distinct hyperplanes $D_{I_1},D_{I_2},D_{I_3}\in \Bc(n,k,\Ac^0)$ such that
  \begin{equation*}
    \codim \left( D_{I_1}\cap D_{I_2}\cap D_{I_3} \right) = 2 .
  \end{equation*}

  We may also require that no further hyperplane contains the intersection $D_{I_1}\cap D_{I_2}\cap D_{I_3}$.
  Indeed, for each $I\in\binom{[n]}{k+1}\setminus\{I_1,I_2,I_3\}$, the condition $D_{I_1}\cap D_{I_2}\cap D_{I_3}\subset D_I$ is equivalent to $\alpha_I\in\operatorname{span}(\alpha_{I_1},\alpha_{I_2},\alpha_{I_3})$.
  This defines a proper Zariski closed subset of $V^\circ_{(\Tb_n,2)}$.
  Hence, since there are only finitely many such $I$, we can choose a real point of $V^\circ_{(\Tb_n,2)}$ satisfying the genericity conditions and avoiding all these closed subsets.
  For this choice, we have
  \begin{equation*}
    \Bc(n,k,\Ac^0)_{D_{I_1}\cap D_{I_2}\cap D_{I_3}}=\{D_{I_1},D_{I_2},D_{I_3}\}.\qedhere
  \end{equation*}
\end{proof}

\begin{proof}[Proof of \zcref{lem:non-very-2}]
  We define a map $c:[r-1]\to\{r,r+1,r+2\}$ as follows:
  \begin{itemize}
    \item if $r-1$ is odd, then
          \begin{align*}
            c(1) & =r,                                \\
            c(i) & =r+1\quad (i\ge 2\text{ is even}), \\
            c(i) & =r+2\quad (i\ge 3\text{ is odd});
          \end{align*}

    \item if $r-1=4$, then $c(1)=c(3)=5$, $c(2)=6$, and $c(4)=7$;

    \item if $r-1\ge 6$ is even, then
          \begin{align*}
            c(1)=c(3) & =r,                                \\
            c(i)      & =r+1\quad (i\ge 2\text{ is even}), \\
            c(i)      & =r+2\quad (i\ge 5\text{ is odd}).
          \end{align*}
  \end{itemize}
  Let $\eta_{n}\coloneq [n]\setminus [r+2]=\{r+3,\dots,n\}$.
  We define
  \begin{align*}
    I^{c}_{i}   & \coloneq \{i,i+1,c(i)\}\sqcup \eta_{n}\qquad (1\le i\le r-2), \\
    I^{c}_{r-1} & \coloneq \{r-1,1,c(r-1)\}\sqcup \eta_{n},                     \\
    I^{c}_{r}   & \coloneq \{r,r+1,r+2\}\sqcup \eta_{n}.
  \end{align*}
  and $\Tb_{c}\coloneq\{ I_{i}^{c}\mid 1\le i\le r\}$. Note that $I^{c}_{i}\in\binom{[n]}{k+1}$ and $\abs{I^c_i\cap I^c_j}\le k-1$ for any distinct $i,j$.
  Thus, $\Tb_c$ satisfies \eqref{item:Q1} and \eqref{item:Q2} and we can define $V_{(\Tb_c,r-1)}^\circ$.
  Since $\sum_{T\in \Tb_c}(\abs{T}-k)=\abs{\bigcup_{T\in\Tb_c}T}-k$, we have $V_{(\Tb_c,r-1)}^\circ\neq \operatorname{Gr}_k^\circ(\Cb^n)$ by \zcref{thm:BB-poset}.

  The $r\times r$ submatrix $A_{\Tb_c}|_{[r]}$ of $A_{\Tb_c}(\Ac^0)=(\alpha_{I^c_i})_{1\le i\le r}$ consisting of the first $r$ columns is as follows:
  \begin{equation*}
    {\hspace{-3em}
      \begin{pmatrix}\Delta_{2c(1)\eta_n}&-\Delta_{1c(1)\eta_n}&0&&&\cdots&0&* \\ 0&\Delta_{3c(2)\eta_n}&-\Delta_{2c(2)\eta_n}&0&&\cdots&0&* \\ 0&0&\Delta_{4c(3)\eta_n}&-\Delta_{3c(3)\eta_n}&0&\cdots&0&* \\ \vdots&\vdots&\ddots&\ddots&\ddots&\ddots&\vdots&\vdots \\ \vdots&\vdots&&\ddots&\ddots&\ddots&0&* \\ 0&0&\cdots&&0&\Delta_{r-1\ c(r-2)\eta_n}&-\Delta_{r-2\ c(r-2)\eta_n}&* \\ \Delta_{r-1\ c(r-1)\eta_n}&0&\cdots&&&0&-\Delta_{1c(r-1)\eta_n}&* \\ 0&0&\cdots&&&&0&\Delta_{r+1\ r+2\ \eta_n}\end{pmatrix}}.
  \end{equation*}
  Thus, the defining equation of $V^{\circ}_{(\Tb_c,r-1)}$ is
  \begin{equation}\label{eq:defining-eq-of-Tc}
    \det A_{\Tb_c}|_{[r]}/\Delta_{r+1\ r+2\ \eta_n}
    =-\left(\prod_{i=1}^{r-2}\Delta_{i+1\ c(i)\eta_n}\right)\Delta_{1\ c(r-1)\eta_n}+ \left(\prod_{i=1}^{r-1}\Delta_{i\ c(i)\eta_n}\right)
  \end{equation}
  and this is non-zero by $V_{(\Tb_c,r-1)}^\circ\neq \operatorname{Gr}_k^\circ(\Cb^n)$.
  Since this is linear in $\alpha_1$, after fixing the other normal vectors generically, as in the discussion of \zcref{lem:non-very-3}, this has a real solution, which implies $V^{\circ}_{(\Tb_c,r-1)}\cap \operatorname{Gr}^\circ_{k}(\Rb^n)\neq \emptyset$.
  Thus, we can take $\Ac^0\in V^{\circ}_{(\Tb_c,r-1)}\cap \mathrm{Gr}_{k}(\Rb^{n})$.
  For this choice, the determinant vanishes, hence the corresponding hyperplanes $D_{I^c_1},\dots,D_{I^c_r}\in \Bc(n,k,\Ac^0)$ satisfy $\codim \bigcap_{j=1}^{r}D_{I^c_j}\le r-1$.

  Next, we show that every proper subset is linearly independent.
  By the form of $A_{\Tb_c}|_{[r]}$, if we delete any one row indexed by $I^c_i$ for $i\in [r-1]$, then the remaining $(r-1)\times (r-1)$ submatrix is triangular with nonzero diagonal entries after a suitable reordering of the rows and columns.
  Thus, the vectors $\alpha_{I^c_j}$ ($j\in[r]\setminus\{i\}$) are linearly independent for any $1\le i\le r-1$.
  This implies the rank of $\{\alpha_{I^c_1},\dots,\alpha_{I^c_r}\}$ is at least $r-1$.
  Thus, we have $\codim \bigcap_{j=1}^{r}D_{I^c_j}= r-1$.

  By definition of $c$, if $r-1$ is odd, then the nonzero entry in the $r$-th column of $A_{\Tb_c\setminus\{ I^c_r\}}$ is unique, and thus the rank of $A_{\Tb_c\setminus\{ I^c_r\}}$ is $r-1$.
  Suppose $r-1$ is even. The condition $\operatorname{rank}A_{\Tb_c\setminus\{ I^c_r\}}= r-2$ holds if and only if there is a unique nonzero vector $d=(d_{1},\dots,d_{r-1})$, up to nonzero scalar multiplication, such that $\sum_{i=1}^{r-1}d_{i}\alpha_{I^c_i}=0$.
  In particular, focusing on the second, third, and the $r$-th columns, we have
  \begin{align*}
    -d_{1}\Delta_{1c(1)\eta_n}+d_{2}\Delta_{3c(2)\eta_n}=-d_{2}\Delta_{2c(2)\eta_n}+d_{3}\Delta_{4c(3)\eta_n}=d_{1}\Delta_{12\eta_n}+d_{3}\Delta_{34\eta_n}=0,
  \end{align*}
  and hence, by eliminating the $d_{i}$, the following relation holds
  \begin{equation*}
    \Delta_{3c(2)\eta_n}\Delta_{4c(3)\eta_n}\Delta_{12\eta_n}-\Delta_{1c(1)\eta_n}\Delta_{2c(2)\eta_n}\Delta_{34\eta_n}=0.
  \end{equation*}
  On the other hand,
  The above equation is not a consequence of \eqref{eq:defining-eq-of-Tc}: in the Pl\"ucker grading, the index $r+2$ occurs only in \eqref{eq:defining-eq-of-Tc}.
  Hence a generic point of $V^\circ_{(\Tb_c,r-1)}$ avoids this additional equation.
  Thus, they also satisfy $\codim \bigcap_{j\in J}D_{I^c_j}=\abs{J}$ for any proper subset $J\subsetneq [r]$.

  We further choose $\Ac^0\in V^{\circ}_{(\Tb_c,r-1)}\cap \mathrm{Gr}_{k}(\Rb^{n})$ so that no hyperplane other than $D_{I^c_1},\dots,D_{I^c_r}$ contains the intersection $\bigcap_{j=1}^r D_{I^c_j}$.
  For each $I\in\binom{[n]}{k+1}\setminus\{I^c_1,\dots,I^c_r\}$, the condition $\bigcap_{j=1}^r D_{I^c_j}\subset D_I$ is equivalent to $\alpha_I\in\operatorname{span}(\alpha_{I^c_1},\dots,\alpha_{I^c_r})$.
  This defines a proper Zariski closed subset of $V^\circ_{(\Tb_c,r-1)}$.
  Since there are only finitely many such $I$, we may choose a real point of $V^\circ_{(\Tb_c,r-1)}\cap\operatorname{Gr}_k^\circ(\Rb^n)$ avoiding all of them.

  Combining these parts, we conclude that the arrangement $\Ac^0$ has the required property.
\end{proof}

\subsection{A connection to matroid strata}\label{subsec:matroid_strata}
In this final section, we discuss the relationship between the subvarieties $V^\circ_{(\Tb,r)}$ and matroid strata, also called realization spaces of matroids.
The content of this section consists of further remarks and is independent of the proof of the main result.
We show that, under certain conditions, an appropriate projection of a matroid stratum coincides with $V^\circ_{(\Tb,r)}$.
The basic idea is to treat the Manin--Schechtman arrangement as an induced arrangement in a fiber of a vector bundle over the Grassmannian.
This appears in \cite{Fal94, ABFKST23}, but we recall it here.

\subsubsection{Affine hyperplane arrangements}\label{subsubsec:affine}
We first consider the moduli of affine arrangements.
As in the central case, an essential affine arrangement with $n$ hyperplanes in $\Cb^k$ corresponds to a $k$-dimensional affine subspace of $\Cb^n$.
For an essential affine arrangement $\Ac^t=\{H_1^{t_1},\dots,H_n^{t_n}\}$ in $\Cb^k$, we write
\begin{equation*}
  W(\Ac^t)\coloneq\operatorname{Im}(\alpha_1,\dots,\alpha_n)-t\subset\Cb^n.
\end{equation*}
Then $W(\Ac^t)\cap\{x_i=0\}$ corresponds to the affine hyperplane $H_i^{t_i}=\alpha_i^{-1}(t_i)$ in $W(\Ac^t)$.
Since $\Ac^t$ is essential, $\alpha=(\alpha_1,\ldots,\alpha_n)$ is injective, and hence $W(\Ac^t)$ is a $k$-dimensional affine subspace of $\Cb^n$.
Conversely, a $k$-dimensional affine subspace $W\subset\Cb^n$ determines an affine hyperplane arrangement in $W$ by intersecting it with the coordinate hyperplanes $\{x_i=0\}$.
Thus, the moduli space of affine arrangements consisting of $n$ hyperplanes in $\Cb^k$ is identified with the space of $k$-dimensional affine subspaces of $\Cb^n$, which we denote by $\operatorname{GrAff}_{k}(\Cb^n)$.

The \emph{cone} over $\Ac^t$ is the central arrangement $c\Ac^t=\{cH_1,\ldots,cH_n,H_{n+1}\}$, where $cH^{t_i}_i\coloneq\{(x,x_{n+1})\in\Cb^{k+1}\mid \alpha_i(x)-t_i x_{n+1}=0\}$ and $H_{n+1}=\{(x,x_{n+1})\in\Cb^{k+1}\mid x_{n+1}=0\}$.
This operation corresponds to the embedding $\operatorname{GrAff}_{k}(\Cb^n)\hookrightarrow \operatorname{Gr}_{k+1}(\Cb^{n+1})$ given by $W\mapsto \widehat W\coloneq\operatorname{span}\{(w,1)\in\Cb^{n+1}\mid w\in W\}$.
Its image is $\{\widehat{W}\in \operatorname{Gr}_{k+1}(\Cb^{n+1})\mid \widehat{W}\not\subset \xi_{n+1}\}$,
where $\xi_{n+1}\coloneq\{x\in\Cb^{n+1}\mid x_{n+1}=0\}$.

We define the map
\begin{equation*}
  \gamma\colon\operatorname{GrAff}_{k}(\Cb^n)\to \operatorname{Gr}_{k}(\Cb^n)
\end{equation*}
by $\gamma(W)\coloneq \widehat W\cap \xi_{n+1}$ where we identify $\Cb^n\cong \xi_{n+1}\subset\Cb^{n+1}$.
Since $W(\Ac^t)=-t+W(\Ac)$, we have
\begin{equation*}
  \gamma\bigl(W(\Ac^t)\bigr)=\widehat{W(\Ac^t)}\cap \xi_{n+1}=W(\Ac).
\end{equation*}
Thus $\gamma$ sends an affine arrangement to its underlying central arrangement.
Note that, $\gamma$ corresponds to taking the restriction $c\Ac^{H_{n+1}}$ of $c\Ac^t$ to $H_{n+1}$, i.e. $\gamma(W(\Ac^t))={W(c\Ac^t)}\cap \xi_{n+1}=W((c\Ac^t)^{H_{n+1}})$.

\begin{prop}
  The map $\gamma\colon\operatorname{GrAff}_{k}(\Cb^n)\longrightarrow\operatorname{Gr}_{k}(\Cb^n)$ is a vector bundle.
  More precisely, if $\Sc$ denotes the tautological subbundle on $\operatorname{Gr}_{k}(\Cb^n)$, then $\gamma$ is identified with the total space of the quotient bundle $\Cb^n/\Sc$.
\end{prop}
\begin{proof}
  A $k$-dimensional affine subspace $W\subset\Cb^n$ is determined by its direction $\gamma(W)\in\operatorname{Gr}_{k}(\Cb^n)$ and a translation parameter in $\Cb^n/\gamma(W)$.
  Indeed, if $W=v+\gamma(W)$, then $v$ is defined only modulo $\gamma(W)$.
  Hence the fiber over $W\in\operatorname{Gr}_{k}(\Cb^n)$ is $\gamma^{-1}(W)\simeq \Cb^n/W$, which is the fiber of $\Cb^n/\Sc$ over $W$.
  Therefore, $\gamma$ is a vector bundle of rank $n-k$.
\end{proof}

\subsubsection{Matroids and their strata}\label{subsubsec:matroid}
We recall the description of matroids in terms of flats (cf. \cite{Oxl11}) and matroid strata (cf. \cite{GGMS87,BLVS99}).
A \emph{matroid} $M=([n],\Fc)$ is a pair consisting of a finite set $[n]$, called the \emph{ground set}, and a family $\Fc=\Fc(M)\subset 2^{[n]}$ of subsets of $[n]$ satisfying:
\begin{itemize}
  \item $[n]\in \Fc$,
  \item $F_1,F_2\in\Fc$ implies $F_1\cap F_2\in \Fc$,
  \item If $F\in \Fc$ and $\{F_1,F_2,\dots,F_m\}$ is the set of flats covering $F$ with respect to inclusion, then $\{F_1\setminus F,F_2\setminus F,\dots,F_m\setminus F\}$ forms a partition of $[n]\setminus F$.
\end{itemize}
For a matroid $M=([n],\Fc)$, elements of $\Fc$ are called flats.
The family of flats of a matroid has a structure of the lattice ordered by inclusion.

The lattice $\Fc(M)$ of flats of a matroid $M$ is an abstraction of the intersection lattice of a hyperplane arrangement.
Indeed, if we set $\Fc_\Ac=\{F\mid \{H_i\mid i\in F\}=\Ac_X, X\in L(\Ac)\}$ for a central hyperplane arrangement $\Ac=\{H_1,\dots,H_n\}$, then $([n],\Fc_\Ac)$ is a matroid.
A matroid obtained from a hyperplane arrangement in this way is called \emph{linear}.
A matroid isomorphic to a linear matroid is called \emph{representable}.
As another example, the \emph{uniform} matroid $U_{k,n}$ on $n$ elements is defined by $\Fc(U_{k,n})=\{F\in 2^{[n]}\mid \abs{F}\le k-1\}\cup \{[n]\}$, which corresponds to the generic arrangement consisting of $n$ hyperplanes in $\Cb^k$.
A matroid $M$ on $[n]$ is \emph{paving} if $\operatorname{cl}_M S\in \Hc(M)$ for any $S\in \binom{[n]}{r(M)-1}$ where $\operatorname{cl}_M S\coloneq\bigcap_{\substack{F\in \Fc,\ F\supset S}}F$.

Let $\Hc=\Hc(M)$ be the set of maximal elements of $\Fc\setminus \{[n]\}$ with respect to inclusion for a matroid $M=([n],\Fc)$.
The elements of $\Hc$ are usually called hyperplanes of $M$, but in this paper, to avoid confusion, we call them \emph{coatoms} of $M$.
Note that the set of coatoms of a matroid uniquely determines the matroid.
The rank $r(M)$ of $M$ is defined to be the length of a maximal chain $F_0\subset F_1\subset\dots\subset F_{r(M)}$ of flats of $M$.

The contraction $M/i$ of $M$ by $i\in [n]$ is the matroid on $[n]\setminus\{i\}$ whose family of flats is $\Fc(M/i)=\{F\setminus \{i\}\mid i\in F\in \Fc\}$.
The contraction of matroids corresponds to the restriction of hyperplane arrangements, provided that restrictions are regarded as indexed arrangements.
That is, for a hyperplane arrangement $\Ac=\{H_1,\dots,H_n\}$, we have $([n-1],\Fc_{\Ac^{H_{n}}})=([n],\Fc_\Ac)/n$.

For a rank $k$ matroid $M$ on $[n]$, the \emph{matroid stratum} of $M$ is
\begin{equation*}
  R_M\coloneq\{W(\Ac)\in\operatorname{Gr}_k(\Cb^n)\mid  ([n],\Fc_\Ac)=M\}.
\end{equation*}
These strata are often referred to as realization spaces of matroids in the Grassmannian.

\subsubsection{Matroid strata and the essentialization of the Manin--Schechtman arrangement}
Fix a point $\Ac\in\operatorname{Gr}_{k}(\Cb^n)$ and assume that $\Ac$ is generic.
Then we have
\begin{equation*}
  \gamma^{-1}(\Ac)=\gamma^{-1}(W(\Ac))\simeq\Cb^n/W(\Ac).
\end{equation*}
On the other hand, we have $\Sb(\Ac)\simeq \Cb^n$.
If $t$ and $t+w$ differ by an element $w\in W(\Ac)$, then the corresponding affine subspaces $W(\Ac^t)=-t+W(\Ac)$, and $ W(\Ac^{t+w})=-(t+w)+W(\Ac)$ coincide.
Thus, on the fiber of the affine Grassmannian, $t$ is recorded only as an element of $\Cb^n/W(\Ac)\simeq \Sb(\Ac)/D_{[n]}$.

Let $R_M$ denote the matroid stratum in $\operatorname{Gr}_{k+1}(\Cb^{n+1})$ corresponding to a matroid $M$ of rank $k+1$ on $[n+1]$.
Via the above embedding, $\operatorname{GrAff}_{k}(\Cb^n)$ is an open subvariety of $\operatorname{Gr}_{k+1}(\Cb^{n+1})$.
Thus, we can consider $\gamma^{-1}(\Ac)\cap R_M$ for each $M$.
This is the set of affine arrangements $\Ac^t$ with fixed underlying central arrangement $\Ac$.

Since $\Ac$ is generic, a general point in the fiber $\gamma^{-1}(\Ac)$ corresponds to the uniform matroid $U_{k+1,n+1}$ after coning.
Hence, a general point belongs to the maximal matroid stratum.
Points outside the maximal stratum correspond to non-generic central arrangements $c\Ac^t=\{cH_1,\ldots,cH_n,H_{n+1}\}$.
More concretely, since $\Ac$ is generic, the minimal degeneracy condition is that there exist $k+1$ affine hyperplanes passing through a common point.
For each such subset, this condition is given by a linear equation in $t\in\Cb^n$.
The hyperplane arrangement in $\Cb^n$ obtained from these linear equations is precisely the Manin--Schechtman arrangement of $\Ac$, as defined in \zcref{sec:MSnk-BBA-poset}.

Moreover, each defining equation of this Manin--Schechtman arrangement is invariant under translations in the direction $W(\Ac)$.
Indeed, any $w\in W(\Ac)$ can be written as $w=(\alpha_1(x),\ldots,\alpha_n(x))$ for some $x\in\Cb^k$.
Replacing $t$ by $t+w$ corresponds to translating all hyperplanes $H_i^{t_i}$ simultaneously by $w$.
Hence, their intersection poset does not change.
The arrangement obtained on the quotient $\gamma^{-1}(\Ac)\simeq \Cb^n/W(\Ac)\simeq \Sb(\Ac)/D_{[n]}$ is the essentialization of the Manin--Schechtman arrangement.

Consequently, the codimension-one part of the stratification induced on the fiber $\gamma^{-1}(\Ac)$ by the intersections with the matroid strata agrees with the essentialization of the Manin--Schechtman arrangement of $\Ac$.

\begin{prop}
  Let $\Ac\in\operatorname{Gr}_{k}(\Cb^n)$ be a generic hyperplane arrangement.
  Then under the identification $\gamma^{-1}(\Ac)\simeq\Cb^n/W(\Ac)$, the codimension-one part of the stratification induced by the intersections
  \begin{equation*}
    \gamma^{-1}(\Ac)\cap R_M
  \end{equation*}
  with the matroid strata $R_M\subset\operatorname{Gr}_{k+1}(\Cb^{n+1})$ agrees with the essentialization of the Manin--Schechtman arrangement of $\Ac$.
\end{prop}

\subsubsection{$V^\circ_{(\Tb,r)}$ as the closure of the image of a matroid stratum}

We now explain how the varieties $V^\circ_{(\Tb,r)}$ are related to projections of matroid strata.
Let $\Tb\subset 2^{[n]}$ be an antichain satisfying \eqref{item:Q1} and \eqref{item:Q2}.
We associate to $\Tb$ the rank $k+1$ paving matroid $M(\Tb)=([n+1],\Hc(\Tb))$ whose coatoms are
\begin{equation*}
  \Hc({\Tb})= \Tb \cup \left\{T'\in \binom{[n]}{k}\mathrel{}\middle|\mathrel{} T'\not\subset T \text{ for all } T\in\Tb \right\} \cup \left\{ T'\cup\{n+1\} \mathrel{}\middle|\mathrel{}T'\in \binom{[n]}{k-1} \right\}.
\end{equation*}
Let $ R^{\operatorname{aff}}_{M(\Tb)} =R_{M(\Tb)}\cap \operatorname{GrAff}_{k}(\Cb^n)$ be the corresponding affine matroid stratum. We consider its image under $  \gamma\colon\operatorname{GrAff}_{k}(\Cb^n) \to \operatorname{Gr}_{k}(\Cb^n)$.

Let $\Ac\in \mathrm{Gr}_k^\circ(\Cb^n)$ be a generic arrangement.
Under the identification $\gamma^{-1}(\Ac)\simeq\Sb(\Ac)/D_{[n]}$, the condition that a subset $I\subset[n]$ with $k+1$ elements is dependent in $c\Ac^t$ is equivalent to $\Ac^t\in D_I(\Ac)$.
Equivalently, the affine hyperplanes $\{H_i^{t_i}\mid i\in I\}$ have a non-empty intersection.

Let $\mathbb{J}(\Tb)\coloneq\{I\in \binom{[n]}{k+1}\mid I\not\subset T \text{ for all }T\in\Tb\}$.
Then the set
\begin{equation*}
  \bigcap_{T\in\Tb}D_T(\Ac)\setminus\bigcup_{I\in\mathbb{J}(\Tb)}D_I(\Ac)
\end{equation*}
consists of those translates for which exactly the dependencies among $(k+1)$-subsets of $[n]$ forced by $\Tb$ occur.
Therefore, whenever this set is non-empty, it is equal to $\gamma^{-1}(\Ac)\cap R^{\operatorname{aff}}_{M(\Tb)}$.
In particular, its closure in the fiber $\gamma^{-1}(\Ac)$ is $\bigcap_{T\in\Tb}D_T(\Ac)$.

The condition $\codim\bigcap_{T\in\Tb}D_T(\Ac)=r$ is open inside $V^\circ_{(\Tb,r)}$.
Moreover, for each $I\in\mathbb{J}(\Tb)$, the condition $\bigcap_{T\in\Tb}D_T(\Ac)\not\subset D_I(\Ac)$ is also open.
Indeed, its failure is equivalent to the inclusion of the normal vector of $D_I(\Ac)$ in $\left(\bigcap_{T\in\Tb}D_T(\Ac)\right)^\perp=\sum_{T\in\Tb}D_T(\Ac)^\perp$.
Since $\mathbb{J}(\Tb)$ is finite, the simultaneous non-containment condition for all $I\in\mathbb{J}(\Tb)$ is open.

\begin{thm}\label{prop:V-as-closure-of-matroid-stratum}
  Let $\Tb\subset 2^{[n]}$ be an antichain satisfying \eqref{item:Q1} and \eqref{item:Q2}.
  Suppose that there exists a generic arrangement $\Ac_0$ and an intersection $X\in L(\Bc(n,k,\Ac_0))$ such that $\Tb=\Tb(X)$.
  Let
  \begin{equation*}
    r=\max\{\codim X\mid \Tb=\Tb(X),X\in L(\Bc(n,k,\Ac)),\Ac\in \operatorname{Gr}_k^\circ(\Cb^n)\}.
  \end{equation*}
  If $V^\circ_{(\Tb,r)}$ is irreducible, then we have
  \begin{equation*}
    \overline{\gamma\left(R^{\operatorname{aff}}_{M(\Tb)}\right)}=V^\circ_{(\Tb,r)}.
  \end{equation*}
  Here the closure is taken in $\operatorname{Gr}_{k}^{\circ}(\Cb^n)$.
\end{thm}
\begin{proof}
  By the fiberwise description above, if $\Ac\in \gamma(R^{\operatorname{aff}}_{M(\Tb)})$, then $\bigcap_{T\in\Tb}D_T(\Ac)$ has canonical presentation $\Tb$.
  Hence, by the definition of $r$, its codimension is at most $r$.
  Thus $\gamma(R^{\operatorname{aff}}_{M(\Tb)})\subset V^\circ_{(\Tb,r)}$, and the same inclusion holds after taking closures.

  Define
  \begin{equation*}
    U_\Tb=\left\{\Ac\in V^\circ_{(\Tb,r)}\mathrel{}\middle|\mathrel{}\codim\bigcap_{T\in\Tb}D_T(\Ac)=r\text{ and }\bigcap_{T\in\Tb}D_T(\Ac)\not\subset D_I(\Ac)\text{ for all } I\in\mathbb{J}(\Tb)\right\}.
  \end{equation*}
  The preceding paragraph shows that $U_\Tb$ is open in $V^\circ_{(\Tb,r)}$.
  It is non-empty by the definition of $r$, because $\Tb$ occurs as the canonical presentation of an intersection of codimension $r$.
  Since $V^\circ_{(\Tb,r)}$ is irreducible, $U_\Tb$ is dense in $V^\circ_{(\Tb,r)}$.

  Again by the fiberwise description, every $\Ac\in U_\Tb$ belongs to $\gamma(R^{\operatorname{aff}}_{M(\Tb)})$.
  Hence $U_\Tb\subset \gamma(R^{\operatorname{aff}}_{M(\Tb)})$ holds.
  Therefore, we have $V^\circ_{(\Tb,r)}=\overline{U_\Tb}\subset\overline{\gamma(R^{\operatorname{aff}}_{M(\Tb)})}\subset V^\circ_{(\Tb,r)}$.
  This proves the assertion.
\end{proof}

By combining this with \zcref{thm:BB-poset} and \zcref{prop:hypersurface}, we obtain the following two corollaries.

\begin{cor}\label{cor:very-generic-strata}
  Let $\Tb\in P(n,k)$ and let $r=\sum_{T\in\Tb}(\abs{T}-k)$.
  Then we have $  \overline{\gamma\left(R^{\operatorname{aff}}_{M(\Tb)}\right)}={V^\circ_{(\Tb,r)}}=\mathrm{Gr}_k^\circ(\Cb^n)$.
\end{cor}
\begin{proof}
  For every $\Ac\in \mathrm{Gr}_k^\circ(\Cb^n)$, we have $\codim \bigcap_{T\in\Tb}D_T(\Ac)\le r$.
  Hence $V^\circ_{(\Tb,r)}=\mathrm{Gr}_k^\circ(\Cb^n)$.
  Moreover, by \zcref{thm:BB-poset}, for a very generic arrangement $\Ac$, the element $\Tb\in P(n,k)$ corresponds to $X=\bigcap_{T\in\Tb}D_T(\Ac)$ with $\codim X=\sum_{T\in\Tb}(\abs{T}-k)=r$ and $\Tb(X)=\Tb$.
  Since $\mathrm{Gr}_k^\circ(\Cb^n)$ is irreducible, \zcref{prop:V-as-closure-of-matroid-stratum} gives the claim.
\end{proof}

\begin{cor}\label{cor:non-very-generic-hypersurface-strata}
  Let $\Tb\subset 2^{[n]}$ be an antichain satisfying \eqref{item:Q1} and \eqref{item:Q2}.
  Suppose that $\sum_{T\in\Tb}(|T|-k) = \abs{\bigcup_{T\in\Tb}T}-k = r+1.$
  If $V^\circ_{(\Tb,r)}$ is non-empty and irreducible, then we have $\overline{\gamma\left(R^{\operatorname{aff}}_{M(\Tb)}\right)}=V^\circ_{(\Tb,r)}$.
  In particular, $V^\circ_{(\Tb,r)}$ is a hypersurface in $\mathrm{Gr}_k^\circ(\Cb^n)$.
\end{cor}
\begin{proof}
  By \zcref{prop:hypersurface}, $V^\circ_{(\Tb,r)}$ is a hypersurface defined by a non-zero $(r+1)$-minor of $A_\Tb(\Ac)$.
  The stronger rank condition $\operatorname{rank}A_\Tb(\Ac)\le r-1$
  is a proper closed condition on this hypersurface, since some
  $r$-minor of $A_\Tb(\Ac)$ is not identically zero on it.

  Moreover, for each $I\in\binom{[n]}{k+1}$ not contained in any member of $\Tb$, the additional containment $\bigcap_{T\in\Tb}D_T(\Ac)\subset D_I(\Ac)$ is a closed condition and does not hold identically on $V^\circ_{(\Tb,r)}$.
  Indeed, if $I\not\subset\bigcup_{T\in\Tb}T$, this is clear from the supports of the corresponding normal vectors.
  If $I\subset\bigcup_{T\in\Tb}T$, it follows by replacing one row in the determinant used in the proof of \zcref{prop:hypersurface} by $\alpha_I$; the resulting determinant is not a scalar multiple of the defining equation of $V^\circ_{(\Tb,r)}$.

  Since there are only finitely many such $I$ and
  $V^\circ_{(\Tb,r)}$ is irreducible, the complement of these proper closed subsets is non-empty and Zariski open.
  On this open subset, the codimension is $r$, and no additional component occurs, hence the canonical presentation is $\Tb$.

  Therefore, there exists a non-empty Zariski open subset of $V^\circ_{(\Tb,r)}$ on which $\bigcap_{T\in\Tb}D_T(\Ac)$ has codimension $r$ and canonical presentation $\Tb$.
  Thus \zcref{prop:V-as-closure-of-matroid-stratum} applies and gives the desired equality.
\end{proof}

\zcref{cor:very-generic-strata} and \ref{cor:non-very-generic-hypersurface-strata} can be restated in terms of matroid strata as follows.
\begin{cor}
  Let $M=([n+1],\Fc)$ be a representable paving matroid of rank $k+1$, and let
  $\Tb=\{F\in \Hc(M)\mid \abs{F}>k\}$.
  Suppose that
  \begin{equation*}
    \left\{T'\cup\{n+1\}\mathrel{}\middle|\mathrel{}T'\in\binom{[n]}{k-1}\right\}\subset \Hc(M).
  \end{equation*}
  Then the following holds.
  \begin{itemize}
    \item If $\Tb\in P(n,k)$, then we have $  \overline{\gamma\left(R^{\operatorname{aff}}_{M}\right)}=\mathrm{Gr}_k^\circ(\Cb^n)$.
    \item Suppose that
          $\sum_{T\in\Tb}(\abs{T}-k)=\abs{\bigcup_{T\in\Tb}T}-k=r+1$.
          If $V^\circ_{(\Tb,r)}$ is non-empty and irreducible, then $\overline{\gamma\left(R^{\operatorname{aff}}_{M}\right)}$ is a hypersurface in $\mathrm{Gr}_k^\circ(\Cb^n)$.
  \end{itemize}
\end{cor}

\begin{proof}
  The assumptions on the coatoms imply $M=M(\Tb)$.
  Thus, the first assertion follows from \zcref{cor:very-generic-strata}, and the second assertion follows from \zcref{cor:non-very-generic-hypersurface-strata}.
\end{proof}
\bibliographystyle{amsalpha}
\bibliography{non_Kpi1}
\end{document}